\documentclass[10pt,a4paper,reqno]{amsart}
\usepackage[utf8]{inputenc}
\usepackage{amsmath}
\usepackage{mathpazo}
\usepackage{enumitem}
\usepackage{tikz-cd}
\usepackage{amsthm}
\usepackage{indentfirst}
\usepackage{amsfonts}
\usepackage{tikz}
\usepackage{verbatim}
\usepackage{enumitem}
\usepackage{amssymb}
\usepackage[hidelinks,linktocpage]{hyperref}
\hypersetup{
	colorlinks   = true, 
	urlcolor     = blue, 
	linkcolor    = blue, 
	citecolor   = green 
}
\usepackage{array}
\usepackage{stmaryrd}
\usepackage{dsfont}
\usepackage{graphicx} 
\usepackage[left=2.5cm,right=2.5cm,top=3cm,bottom=3cm]{geometry}
\usepackage{color}
\usepackage[all]{xy}
\usepackage[english, french]{babel}
\usepackage{mathrsfs}
\theoremstyle{plain}
\newtheorem{prop}{Proposition}[section]
\newtheorem{lem}[prop]{Lemma}
\newtheorem*{lem*}{Lemma}
\newtheorem{thm}[prop]{Theorem}

\newtheorem{cor}[prop]{Corollary}
\theoremstyle{definition}
\newtheorem{defi}[prop]{Definition}
\newtheorem{ex}[prop]{Example}
\theoremstyle{remark}

\newcommand{\qbin}[2]{\left[\begin{smallmatrix} #1 \\ #2 \end{smallmatrix} \right]}

\newcommand{\s}{\mathfrak S}
\newcommand{\cL}{\mathcal{L}}
\newcommand{\U}{\mathcal U}

\numberwithin{equation}{section}
\begin{document}

\selectlanguage{english}

	\title{Faithfulness of simple 2-representations of $\mathfrak{sl}_2$}
	
	\author{Laurent Vera}
	
	\maketitle
	
	\begin{abstract}
	Let $\mathcal U$ be the 2-category associated with $\mathfrak{sl}_2$. We prove that a complex of 1-morphisms of $\mathcal U$ is null-homotopic if and only if its image in every simple 2-representation is null-homotopic. Under mild boundedness assumptions, we prove that it actually suffices for the image in the simple 2-representations to be acyclic. We apply this result to the study of the Rickard complex $\Theta$ categorifying the action of the simple reflection of $\mathrm{SL}_2$. We prove that $\Theta$ is invertible in the homotopy category of $\U$, and that there is a homotopy equivalence $\Theta E \simeq F\Theta[-1]$.
	\end{abstract}
	
	\tableofcontents
	
	\section{Introduction} Consider the Lie algebra $\mathfrak{sl}_2$ of traceless complex $2\times 2$ matrices, and $U$ its enveloping algebra or associated quantum group. The category of finite dimensional $U$-modules is well understood. It is semi-simple and its simple objects are the $L(n)$ for $n \in \mathbb N$, where $L(n)$ denotes the simple $U$-module of dimension $n+1$. A well-known result is that the collection $\left(L(n)\right)_{n \in \mathbb{N}}$ of $U$-modules is \textit{faithful}. That is, if an element $z$ of $U$ acts by zero on $L(n)$ for all $n \in \mathbb{N}$, then $z=0$. In this paper, we prove a categorification of this result.
	
	The current theory of categorifications of quantized enveloping algebras and their representations began with the seminal work of Chuang and Rouquier \cite{ChR}. In their paper, Chuang-Rouquier introduce a notion of \textit{$\mathfrak{sl}_2$-categorification}, or action of $\mathfrak{sl}_2$ on a category. Roughly speaking, an $\mathrm{sl}_2$-categorification is the data of a category $\mathcal V$, a adjoint pair $(E,F)$ of endofunctors of $\mathcal{V}$ and some natural transformations subject to certain relations. These relations imply in particular that the Grothendieck group of $\mathcal{V}$ inherits an action of $\mathfrak{sl}_2$, with the functors $E$ and $F$ corresponding to the matrices $e = \left(\begin{smallmatrix} 0 & 1 \\ 0 & 0 \end{smallmatrix}\right)$ and $f = \left(\begin{smallmatrix} 0 & 0 \\ 1 & 0 \end{smallmatrix}\right)$ respectively. The finite-dimensional simple $U$-modules admit categorical analogues in this context, called \textit{minimal categorifications}. Unlike the representation theory of $\mathfrak{sl}_2$, the resulting theory is not semi-simple. However, minimal categorifications do form the elementary building blocks of all $\mathfrak{sl}_2$-categorifications, in the sense that every $\mathfrak{sl}_2$-categorification admits a filtration whose quotients are minimal categorifications. This notably enables Chuang-Rouquier to prove strong structural results about $\mathfrak{sl}_2$-categorifications, with deep consequences such as a proof of Brou\'e's abelian defect group conjecture for symmetric groups.
	
	This theory was later expanded in \cite{kl2}, \cite{kl3}, \cite{lauda}, \cite{2km}, \cite{qha}. These papers define a 2-category $\U$ associated with $\mathfrak{sl}_2$ (or more generally associated with any Kac-Moody algebra, but we will only work in type $A_1$ in this paper). This 2-category categorifies the integral idempotent form of $U$. The $\mathfrak{sl}_2$-categorifications from \cite{ChR} are precisely the 2-representations of $\U$, that is 2-functors from $\U$ to the 2-category of categories. Note that the definitions of Khovanov-Lauda and Rouquier, despite being different, are actually equivalent by work of Brundan \cite{brun}.
	
	The minimal categorifications from \cite{ChR} have an additive version introduced in \cite{2km}, called \textit{simple 2-representations} and denoted $\cL(n)$ for $n \in \mathbb{N}$. They form the building blocks of integrable 2-representations of $\U$. For instance, a result of particular interest is \cite[Lemma 5.17]{2km}, which says that if a complex of 1-morphisms of $\U$ is null-homotopic in all simple 2-representations, then it is null-homotopic in all integrable 2-representations. Our main theorem is an extension of this result, showing that we can lift the null-homotopy to the 2-category $\U$ itself.
	
	\begin{thm}[Theorem \ref{lifthom}]\label{main}
		\begin{enumerate}
			\item Let $C$ be a complex of 1-morphisms of $\mathcal U$. If the image of $C$ in $\cL(n)$ is null-homotopic for all $n \in \mathbb N$, then $C$ is hull-homotopic.
			\item Let $f$ be a morphism between complexes of 1-morphisms of $\mathcal U$. If the image of $f$ in $\cL(n)$ is a homotopy equivalence for all $n \in \mathbb N$, then $f$ is a homotopy equivalence.
		\end{enumerate}
	\end{thm}

	Note that the second point results from the first, by considering the cone of $f$. We  now give a short overview of the proof of this Theorem. Our proof relies on Lauda's determination of the indecomposable 1-morphisms of $\U$, as well as bases for 2-morphism spaces 
	\cite{lauda}. We prove that the image of an indecomposable 1-morphism of $\U$ in $\cL(n)$ is indecomposable if $n$ is large enough and of suitable parity. We also prove that the images of two non-isomorphic indecomposable 1-morphisms of $\U$ in $\cL(n)$ are non-isomorphic if $n$ is large enough and of suitable parity. Once these results are established, Theorem \ref{main} follows from general facts about Krull-Schmidt categories.
	
	Our approach is similar to that of \cite{lauda}, which uses 2-representations of $\U$ coming from the geometry of Grassmanians to determine the indecomposable 1-morphisms and bases of 2-morphism spaces of $\U$. Another related work is \cite{bl}, which proves that $\U$ can be realized as an inverse limit of some integrable 2-representations. In particular, this shows that $\U$ is determined by its integrable 2-representation theory.
	
	Using results and methods of \cite{ChR}, we can prove that under mild boundedness assumptions, it suffices for the image of a complex of 1-morphisms of $\U$ in the simple 2-representations to be acyclic to conclude that the complex is null-homotopic.
	
	\begin{thm}[Theorem \ref{liftder}]\label{main2}
		\begin{enumerate}
			\item Let $C$ be a complex of 1-morphisms of $\mathcal U$ such that for any object $M$ of an integrable 2-representation, the complex $C(M)$ is bounded. If the image of $C$ in $\cL(n)$ is acyclic for all $n \in \mathbb N$, then $C$ is null-homotopic.
			\item Let $f$ be a morphism between complexes of 1-morphisms of $\mathcal U$ that both satisfy the boundedness assumption of (1). If  the image of $f$ in $\cL(n)$ is a quasi-isomorphism for all $n \in \mathbb N$, then $f$ is a homotopy equivalence.
		\end{enumerate}
	\end{thm}

	Categorifications of simple modules exist for all Kac-Moody algebras, see \cite{2km} for an abstract approach, and \cite{kk}, \cite{webs} for a concrete approach using cyclotomic quiver Hecke algebras. It would be interesting to generalize Theorems \ref{main} and \ref{main2} to arbitrary types. Unfortunately, our proof does not carry over since we do not have an explicit form for the indecomposable 1-morphisms of $\U$ in general. It is known that the indecomposable 1-morphism of $\U$ decategorify to the canonical basis of $U$ (see \cite{qha}, \cite{vv}), but the canonical basis does not have an explicit expression in general.\\	
	
	Our main motivation for Theorems \ref{main} and \ref{main2} is that performing explicit computations in the homotopy category of $\U$ is in general difficult. Our results provide a strategy to approach such computations. The simple 2-representations have an explicit form: they are categories of free modules over certain polynomial rings. In such a setting, proving that a complex is acyclic is a reasonable task in general.
	
	We give a concrete example of this strategy to study the Rickard complex $\Theta$. To understand our results, let us first recall the decategorified picture. The simple reflection $s=\left(\begin{smallmatrix} 0 & 1 \\ -1 & 0 \end{smallmatrix}\right)$ of $\mathrm{SL}_2$ acts on integrable representations of $U$, providing isomorphisms between weight spaces of opposite weights. The action of $s$ and the Chevalley generators $e,f$ of $U$ are related by the relation $se=-fs$. The Rickard complex $\Theta$ is a complex of 1-morphisms of $\U$ which decategorifies to $s$. Chuang and Rouquier proved in \cite{ChR} that $\Theta$ provides derived equivalences on integrable 2-representations of $\U$. The complex $\Theta$ was also studied in \cite{ck}, \cite{ckl}, with a more geometric framework for 2-representations. In this paper, we prove a categorification of the fact that $s$ is invertible and of the relation $se=-fs$.
	
	\begin{thm}[Theorems \ref{thetainv} and \ref{thetae}]\label{main3}
		The complex $\Theta$ is invertible up to homotopy and there is a homotopy equivalence $\Theta E \simeq F\Theta[-1]$.
	\end{thm} 

	Here, $[-1]$ denotes a homological grading shift. To prove these results, we use Theorem \ref{main2}. More precisely, we define a morphism of complexes $\Theta E\rightarrow F\Theta[-1]$, and we check that it is a quasi-isomorphism in every simple 2-representation. Along the way, we revisit some of the results of \cite{ChR} regarding $\Theta$, in the setup of simple 2-representations rather than minimal categorifications. Namely, we prove that in every simple 2-representation, the cohomology of $\Theta$ is concentrated in top degree, and that the top cohomology is invertible. The fact that the top cohomology is invertible implies, thanks to Theorem \ref{main2}, that $\Theta$ is invertible up to homotopy. Our approach to these results is slightly different from that of Chuang and Rouquier, and is based on constructing explicit bases for the terms of $\Theta$ on simple 2-representations.\\
	
	We now give an overview of the structure of our paper. Section 2 introduces the notation used throughout the paper, and Section 3 defines the 2-category $\U$ and its 2-representations, and gives some of its properties, following \cite{2km} and \cite{lauda}. In Section 4, we recall the definition of the simple 2-representations, and we prove Theorems \ref{main} and \ref{main2}. Finally, Section 5 is devoted to the application of our results to the study of the Rickard complex and the proof of Theorem \ref{main3}
	
	\subsection*{Acknowledgments} I thank my advisor, Rapha\"{e}l Rouquier for suggesting this project and for his guidance and support through its completion.
	
	\bigskip
	
	\section{Notation and definitions}
	
	For the rest of this paper, we fix a field $K$. We will use the symbol $\otimes$ to mean $\otimes_K$.
	
	\subsection*{Graded categories} A \textit{graded category} is a $K$-linear category $\mathcal C$ together with an auto-equivalence $\mathcal C \rightarrow \mathcal C$ called the \textit{shift functor}. We will use the notation $M \mapsto qM$ to denote the shift functor on objects. For instance, the category $\mathrm{Vect}_K$ of graded $K$-vector spaces with homogeneous linear maps is graded. In that case the shift functor is as follows: for $V = \oplus_{k \in \mathbb{Z}} V_k \in \mathrm{Vect}_K$, $qV$ is the graded $K$-vector space defined by $(qV)_k = V_{k-1}$. More generally given a graded $K$-algebra $A$, the category $A\mathrm{-mod}$ of finitely generated graded modules and its full subcategory $A\mathrm{-proj}$ of graded projective modules are graded. Their shift functor is the restriction of that of $\mathrm{Vect}_K$. Given a Laurent polynomial $p = \sum_{\ell \in \mathbb Z} p_{\ell}q^{\ell} \in \mathbb N[q,q^{-1}]$ and $M$ an object of a graded category $\mathcal C$, we put
	\[
		pM = \bigoplus_{\ell \in \mathbb{Z}} q^{\ell}M^{\oplus p_{\ell}}.
	\]
	
	For two objects $M,N$ of a graded category $\mathcal C$ and $k \in \mathbb Z$, we define $\mathrm{Hom}_{\mathcal C}^{k}(M,N) = \mathrm{Hom}_{\mathcal C}(M,q^{-k}N)$. The elements of $\mathrm{Hom}_{\mathcal C}^{k}(M,N)$ are said to be \textit{morphisms of degree $k$} from $M$ to $N$. We can then define a graded $K$-vector space $\mathrm{Hom}^{\bullet}_{\mathcal C}(M,N)$ by
	\[
		\mathrm{Hom}^{\bullet}_{\mathcal C}(M,N) = \bigoplus_{k \in \mathbb{Z}} \mathrm{Hom}_{\mathcal C}^{k}(M,N).
	\]
	We have isomorphisms of graded vector spaces $\mathrm{Hom}^{\bullet}_{\mathcal C}(M,qN)\simeq q \, \mathrm{Hom}^{\bullet}_{\mathcal C}(M,N)$ and $\mathrm{Hom}^{\bullet}_{\mathcal C}(qM,N) \simeq q^{-1}\mathrm{Hom}^{\bullet}_{\mathcal C}(M,N)$.
	
	\subsection*{Graded dimensions} A graded $K$-vector space $V = \oplus_{k \in \mathbb{Z}} V_k$ is said to be \textit{locally finite} if each $V_k$ is finite dimensional. In this case, its \textit{graded dimension} is the formal series defined by
	\[
		\mathrm{grdim}(V) = \sum_{k\in \mathbb Z} \mathrm{dim}(V_k)q^k \in \mathbb{N}((q,q^{-1})).
	\] 
	We have $\mathrm{grdim}(qV)=q \, \mathrm{grdim}(V)$. If $M,N$ are two objects of a graded category $\mathcal C$ and $\mathrm{Hom}^{\bullet}_{\mathcal C}(M,N)$ is locally finite, we put
	\[
		\left< M,N \right> = \mathrm{grdim}\left(\mathrm{Hom}^{\bullet}_{\mathcal C}(M,N)\right).
	\]
	Then we have $\left<M,qN\right>=q\left<M,N\right>$ and $\left<qM,N\right>=q^{-1}\left<M,N\right>$. If $M=N$, we put $\left|M \right| = \left< M,M\right>$.
	
	\subsection*{Krull-Schmidt categories} A $K$-linear category $\mathcal C$ is called \textit{Krull-Schmidt} if every object of $\mathcal C$ is isomorphic to a finite direct sum of objects having local endomorphism algebras. If $\mathcal C$ is Krull-Schmidt, then the indecomposable objects are precisely the objects having local endomorphism algebras, and every object decomposes uniquely as a finite direct sum of indecomposable objects up to reordering of the terms. We say that a 2-category $\mathfrak C$ is Krull-Schmidt if for all objects $\lambda, \mu$ of $\mathfrak{C}$ the category $\mathrm{End}_{\mathfrak C}(\lambda,\mu)$ is Krull-Schmidt. In other words, $\mathfrak{C}$ is Krull-Schmidt if every 1-morphism of $\mathfrak{C}$ decomposes as a finite direct sum of 1-morphisms having local endomorphism algebras.
	
	\subsection*{Homological algebra} Let $\mathcal C$ be a $K$-linear category. We denote by $\mathrm{Comp}(\mathcal C)$ the category of (cochain) complexes of $\mathcal C$. That is, the objects of $\mathrm{Comp}(\mathcal C)$ are of the form
	\[
		M = (\cdots \rightarrow M^r \xrightarrow{d_M^r} M^{r+1} \xrightarrow{d_M^{r+1}} M^{r+2} \rightarrow \cdots )
	\]
	with $d_M^{r+1}d_M^r = 0$ for all $r \in \mathbb Z$. If $M \in \mathrm{Comp}(\mathcal C)$, its homological shift $M[1]$ is defined by $(M[1])^r = M^{r-1}$ with differential $d_{M[1]}^r=-d_M^{r-1}$ for all $r\in \mathbb Z$. Note that when $\mathcal{C}$ is graded, $\mathrm{Comp}(\mathcal C)$ has two compatible gradings: the one coming from $\mathcal C$ and the homological grading.
	
	Given $M,N \in \mathrm{Comp}(\mathcal C)$ and $f : M \rightarrow N$ a morphism of complexes, the \textit{cone} of $f$ is the object $\mathrm{Cone}(f)$ of $\mathrm{Comp}(\mathcal C)$ defined by $\mathrm{Cone}(f)^r = M^{r+1} \oplus N^r$ with differential given by
	\[
		d_{\mathrm{Cone}(f)}^r = \left[ 
		\begin{array}{cc}
			-d_M^{r+1} & 0 \\
			f & d_N^r
		\end{array}
		\right].
	\]
	We denote by $K(\mathcal C)$ the homotopy category of $\mathcal C$. If $\mathcal{C}$ is furthermore abelian, we denote by $D(\mathcal C)$ its derived category. In that case, we denote by $H^r(M)$ the $r^{\mathrm{th}}$ cohomology object of a complex $M$.
	
	\bigskip
	
	\section{Categorified quantum $\mathfrak{sl}_2$}
	
	\subsection{Affine nil Hecke algebras}
	
	We start by recalling the essential facts about affine nil Hecke algebras. Let $\s_n$ be the symmetric group on $n$ letters. It is generated by $s_1,\ldots,s_n$, where $s_i$ denotes the transposition $(i \ i+1)$. The \textit{length} of an element $\omega \in \s_n$ is the smallest integer $r$ such that $\omega=s_{i_1}\ldots s_{i_r}$ for some $i_1,\ldots,i_r \in \lbrace 1, \ldots, n-1\rbrace$. We denote the length of $\omega$ by $l(\omega)$. For $k \leqslant \ell$ two integers in $\lbrace 1,\ldots, n\rbrace$, we denote by $\s_{[k,\ell]}$ the subgroup of $\s_n$ generated by $\left  \{s_i, \, i \in \lbrace k,\ldots,\ell-1\rbrace \right \}$. The longest element of $\s_{[k,\ell]}$ is denoted $\omega_0[k,\ell]$. If $a,b \in \mathbb{N}$ are such that $a+b=n$, we will denote the subgroup $\s_{[1,a]}\times\s_{[a+1,n]}$ of $\s_n$ by $\s_a\times\s_b$.  
	
	\begin{defi}
		The \textit{affine nil Hecke algebra} $H_n$ is the unital $K$-algebra on the generators $x_1,\ldots,x_n$ and $\tau_1,\ldots,\tau_{n-1}$ subject to the following relations for all $i,j \in \lbrace1,\ldots,n\rbrace$ and $k,\ell \in \lbrace 1,\ldots,n-1\rbrace$:
		\begin{enumerate}
			\item $x_ix_j=x_jx_i$,
			\item $\tau_k^2=0$,
			\item $\tau_k x_i - x_{s_k(i)}\tau_k = \delta_{i,k+1} - \delta_{i,k}$,
			\item $\tau_k\tau_{\ell}=\tau_{\ell}\tau_k$ if $\vert k - \ell \vert >1$,
			\item $\tau_{k+1}\tau_k\tau_{k+1}=\tau_k\tau_{k+1}\tau_k$ if $k<n-1$.
		\end{enumerate}
		We fix a grading on $H_n$ with $x_1,\ldots,x_n$ in degree 2 and $\tau_1,\ldots,\tau_{n-1}$ in degree $-2$.
	\end{defi}

	Let $P_n=K[x_1,\ldots,x_n]$. We put a grading on $P_n$ with $x_1,\ldots,x_n$ in degree 2. The symmetric group $\s_n$ acts on $P_n$ by permuting $x_1,\ldots,x_n$. The \textit{Demazure operators} are the $P_n^{\s_n}$-linear operators $\partial_1,\ldots,\partial_{n-1}$ on $P_n$ defined by
	\[
		\partial_i(P) = \frac{P-s_i(P)}{x_{i+1}-x_i}.
	\]
	There is an isomorphism of graded $K$-algebras (see \cite[Proposition 3.4]{2km})
	\begin{equation}\label{reppol}
		\left \{ \begin{array}{rcl}
			H_n & \rightarrow & \mathrm{End}_{P_n^{\s_n}}^{\bullet}(P_n), \\
			\tau_k & \mapsto & \partial_k, \\
			x_k & \mapsto & \text{multiplication by} \ x_k.
		\end{array} \right.
	\end{equation}
	Given $\omega = s_{i_1}\ldots s_{i_r}$ with $l(\omega)=r$, the element $\tau_{i_1}\ldots\tau_{i_r} \in H_n$ only depends on $\omega$ and is denoted by $\tau_{\omega}$. We define similarly $\partial_{\omega}=\partial_{i_1}\ldots\partial_{i_r}$.
	
	As a $P_n^{\s_n}$-module, $P_n$ is free of rank $n!$. The following sets are bases
	\begin{align*}
		&\left \{ x_1^{a_1}\ldots x_n^{a_n}, \ 0\leqslant a_i \leqslant n-i \right \}, \\
		&\left \{ \partial_{\omega}(x_2x_3^2\ldots x_n^{n-1}), \ \omega \in \s_n \right \}.
	\end{align*}
	In particular, $H_n$ is isomorphic the algebra of $(n!)\times(n!)$ matrices with coefficients in $P_n^{\s_n}$. The $H_n$-module $P_n$ with action given by (\ref{reppol}) is the unique indecomposable graded projective $H_n$-module up to isomorphism. Define
	\begin{align*}
		& e_n = x_2x_3^2\ldots x_{n}^{n-1}\tau_{\omega_0[1,n]}, \\
		& e_n'=(-1)^{\frac{n(n-1)}{2}}\tau_{\omega_0[1,n]}x_1^{n-1}x_2^{n-2}\ldots x_{n-1}.
	\end{align*}
	These are orthogonal primitive idempotents of $H_n$. We have isomorphisms of graded $(H_n,P_n^{\s_n})$-bimodules:
	\[
		\left \{
		\begin{array}{rcl}
			P_n & \xrightarrow{\sim} & q^{\frac{n(n-1)}{2}}H_ne_n, \\
			1 & \mapsto & \tau_{\omega_0[1,n]},
		\end{array}
		\right. \quad \text{ and } \quad \left \{		
		\begin{array}{rcl}
			P_n & \xrightarrow{\sim} & H_ne'_n, \\
			1 & \mapsto & e'_n.
		\end{array}
		\right.
	\]
	Hence we have Morita equivalences
	\begin{equation}\label{morita}
		\left \{
		\begin{array}{rcl}
		H_n\mathrm{-mod} & \xrightarrow{\sim} & P_n^{\s_n}\mathrm{-mod},  \\
		M & \mapsto & e_nM,
		\end{array}
		\right. \quad \text{ and } \quad \left \{		
		\begin{array}{rcl}
		H_n\mathrm{-mod} & \xrightarrow{\sim} & P_n^{\s_n}\mathrm{-mod}, \\
		M & \mapsto & e'_nM.
		\end{array}
		\right.
	\end{equation}

	\medskip
	
	\subsection{The 2-category $\mathcal U$}
	
	\subsubsection{Definitions} We now define the 2-category associated with $\mathfrak{sl}_2$, following \cite{2km}. For $k \in \mathbb Z$, we define the quantum integer $\left[k\right]$ by
	\[
		\left[k\right] = \frac{q^k-q^{-k}}{q-q^{-1}}.
	\]
	If $k,n \in \mathbb N$ and $k \leqslant n$, we put 
	\[
		\left[k\right]! = \prod_{\ell=1}^k \left[\ell\right], \quad \left[ \begin{array}{c} n \\ k \end{array}\right] = \frac{[n]!}{[k]![n-k]!}.
	\]
	The 2-category $\U$ will be defined as the idempotent completion of a strict 2-category $\U'$ that we start by defining.
	
	\begin{defi}\label{catsl2}
		Define a strict 2-category $\U'$ by the following data.
	\begin{itemize}
		\item The set of objects of $\U'$ is $\mathbb Z$.
		
		\item Given $\lambda, \lambda' \in \mathbb{Z}$, the 1-morphisms $\lambda \rightarrow \lambda'$ are direct sums of shifts of words of the form $E^{n_1}F^{m_1}\ldots E^{n_r}F^{m_r}$, with $r,n_i,m_i \in \mathbb N$ and $\lambda' - \lambda = 2\sum_i (n_i-m_i)$. Composition of of 1-morphisms is given by concatenation of words, and the identity 1-morphism of $\lambda$ is the empty word, denoted $1_{\lambda}$. When we want to specify that a word $X$ has source $\lambda$ (resp. target $\lambda'$), we will write $X1_{\lambda}$ (resp. $1_{\lambda'}X$).
		
		\item The 2-morphisms of $\U'$ are generated by $x : F \rightarrow F$ of degree $2$, $\tau : F^2 \rightarrow F^2$ of degree $-2$, $\eta : 1_\lambda \rightarrow FE1_{\lambda}$ of degree $1+\lambda$, and $\varepsilon : EF1_{\lambda} \rightarrow 1_{\lambda}$ of degree $1-\lambda$, for all $\lambda \in \mathbb Z$. On these generating 2-morphisms, we impose the following relations:
		\begin{enumerate}
			\item\label{sq} $\tau^2=0$,
			\item\label{xt} $\tau \circ Fx - xF\circ \tau = Fx \circ \tau - \tau \circ xF=F^2$,
			\item\label{br} $\tau F\circ F\tau \circ \tau F = F\tau \circ \tau F \circ F\tau$,
			\item\label{ap} $F = F\varepsilon\circ \eta F$ and $E=\varepsilon E \circ E\eta$,
			\item\label{weyl} for all $\lambda \in \mathbb{Z}$, $\rho_{\lambda}$ is invertible, where $\rho_{\lambda}$ is the 2-morphism defined by
				\[
					\rho_{\lambda} = \left \{\begin{array}{cl}
					\left[ \begin{array}{c} 
					\sigma \\
					\varepsilon \\
					\varepsilon \circ Ex \\
					\vdots \\
					\varepsilon\circ Ex^{\lambda-1}
					\end{array} \right] : EF1_{\lambda} \rightarrow FE1_{\lambda} \oplus \left[\lambda\right]1_{\lambda} & \text{ if } \lambda \geqslant 0, \\
					& \\
					\left[\sigma \quad \eta \quad xE \circ \eta \quad \ldots \quad  x^{1-\lambda}E\circ \eta \right] :  EF1_{\lambda} \oplus [-\lambda]1_{\lambda} \rightarrow FE1_{\lambda} & \text{ if } \lambda \leqslant 0,
					\end{array} \right.
				\]	
			with $\sigma = \varepsilon FE \circ E\tau E \circ EF\eta : EF \rightarrow FE$.
		\end{enumerate}
	\end{itemize}
	\end{defi}

	Let us give some more comments on this definition. Relations (\ref{sq}), (\ref{xt}) and (\ref{br}) imply that the affine nil Hecke algebra $H_n$ acts on $F^n$, as follows
	\[
		\left \{ \begin{array}{rcl}
			H_n & \rightarrow & \mathrm{End}_{\U'}^{\bullet}(F^n), \\
			x_k & \mapsto & F^{k-1}xF^{n-k}, \\
			\tau_k & \mapsto & F^{k-1}\tau F^{n-k-1}.
		\end{array} \right.
	\]
	Relation (\ref{ap}) says that for all $\lambda \in \mathbb{Z}$, we have adjoint pairs $\left(E1_{\lambda}, q^{-1-\lambda}F1_{\lambda+2}\right)$ with unit and counit of adjunction given by $\eta$ and $\varepsilon$. Finally, relation (\ref{weyl}) means that there are additional generating 2-morphisms in $\mathcal{U'}$, defined to be the inverses of the maps $\rho_{\lambda}$.	
	
	\begin{defi}\label{defU}
		The 2-category $\U$ is the idempotent completion of $\U'$. This means that $\U$ has the same set of objects as $\U'$ and that for all $\lambda,\lambda' \in \mathbb{Z}$, the category $\mathrm{Hom}_{\U}(\lambda,\lambda')$ is the idempotent completion of the category $\mathrm{Hom}_{\U'}(\lambda,\lambda')$.
	\end{defi}

	In particular, we can define \textit{divided powers} $F^{(n)}$ and $E^{(n)}$ in $\U$ using idempotents 2-morphisms. More precisely, using the idempotent $e_n \in H_n$, we define
	\[
		F^{(n)}=q^{-\frac{n(n-1)}{2}}e_n(F^n).
	\]
	Then we have an isomorphism $F^n \simeq \left[n\right]!F^{(n)}$. Furthermore for $a,b \in \mathbb{N}$ we have
	\begin{equation}\label{decompdp}
		F^{(a)}F^{(b)} \simeq \qbin{a+b}{a}F^{(a+b)}.
	\end{equation}
	Similarly, the adjunction between $E$ and $F$ provides an action of $H_n^{\mathrm{opp}}$ on $E^n$. Using the idempotent $e_n' \in H_n$, we define
	\[
		E^{(n)}=q^{-\frac{n(n-1)}{2}}(E^n)e_n',
	\]
	where we have written the action of $H_n^{\mathrm{opp}}$ on $E^n$ as a right action of $H_n$. Then there are adjoint pairs $\left(E^{(n)}1_{\lambda}, q^{-n(\lambda+n)}F^{(n)}1_{\lambda+2n}\right)$.\\
	
	Given $\lambda \in \mathbb{Z}$ we define
	\[
		U1_{\lambda}  = \bigoplus_{\lambda' \in \mathbb{Z}} \mathrm{Hom}_{\U}(\lambda,\lambda').
	\]
	It is an additive and idempotent complete category. In $\U1_{\lambda}$ there are isomorphisms
	\begin{align}\label{decomp}
		&E^{(a)}F^{(b)}1_{\lambda} \simeq \bigoplus_{i=0}^{\mathrm{min}\left \{a,b\right \}} \qbin{\lambda+a-b}{i} F^{(b-i)}E^{(a-i)}1_{\lambda} \quad \mathrm{if} \ \lambda \geqslant b-a,\\
		\label{decomp2}&F^{(b)}E^{(a)}1_{\lambda} \simeq \bigoplus_{i=0}^{\mathrm{min}\left \{a,b\right \}} \qbin{-\lambda-a+b}{i} E^{(a-i)}F^{(b-i)}1_{\lambda} \quad \mathrm{if} \ \lambda \leqslant b-a.
	\end{align}
	They can be constructed using the isomorphisms $\rho_{\lambda}$ (see \cite[Lemma 4.14]{2km} for details).
	
	\medskip

	\subsubsection{Indecomposable 1-morphisms}
	
	In \cite{lauda}, Lauda defines a different 2-category. The main differences are that Lauda's version requires an adjoint pair $\left(q^{\lambda+1}F1_{\lambda+2}, E1_{\lambda}\right)$ and the inverses of the maps $\rho_{\lambda}$ are given explicitly in terms of the unit and counit of this second adjunction. However, Brundan proved in \cite{brun} that the 2-categories of Lauda and Rouquier are isomorphic. In particular, the results proved in \cite{lauda} remain valid in $\U$. The main result of interest for this paper is the determination of the indecomposable 1-morphisms of $\U$.
	
	\begin{thm}[\cite{lauda}]\label{indecofU} The 2-category $\mathcal U$ is Krull-Schmidt.	A complete set of pairwise non-isomorphic indecomposable 1-morphisms of $\mathcal U$ is given by
		\[
		\left \{ q^sE^{(a)}F^{(b)}1_{\lambda}, \ \lambda \leqslant b-a, \ s \in \mathbb{Z} \right \} \bigcup \left \{ q^sF^{(b)}E^{(a)}1_{\lambda}, \ \lambda > b-a, \ s \in \mathbb Z \right \}.
		\]
		Furthermore, if $X$ is an indecomposable 1-morphism of $\U$ we have $\left| X\right| \in 1 + q\mathbb{N}[q]$.
	\end{thm}
	
	\medskip
	
	\subsection{2-representations}
	
	A \textit{2-representation} of $\mathcal U$ is a strict 2-functor $\mathcal U \rightarrow \mathfrak{Lin}_K$, where $\mathfrak{Lin}_K$ denotes the strict 2-category of $K$-linear, graded and idempotent complete categories. More explicitly, a 2-representation $\mathcal V$ of $\mathcal U$ is the data of
	\begin{itemize}
		\item categories $\mathcal V_{\lambda} \in \mathfrak{Lin}_K$ for each $\lambda \in \mathbb Z$,
		\item functors $E : \mathcal V_{\lambda} \rightarrow \mathcal V_{\lambda+2}$ and $F : \mathcal V_{\lambda} \rightarrow \mathcal V_{\lambda-2}$ for each $\lambda \in \mathbb Z$,
		\item natural transformations $x : F \rightarrow F$ of degree 2, $\tau : F^2 \rightarrow F^2$ of degree $-2$, $\varepsilon : EF1_{\lambda} \rightarrow 1_{\lambda}$ of degree $1-\lambda$, and $\eta : 1_{\lambda} \rightarrow FE1_{\lambda}$ of degree $1+\lambda$,
	\end{itemize}
	such that the relations of Definition \ref{catsl2} are satisfied. An \textit{abelian 2-representation} is a 2-representation $\mathcal V$ such that for all $\lambda \in \mathbb Z$, the category $\mathcal V_{\lambda}$ is abelian. Note that since the functors $E,F$ are both left and right adjoints, they are exact on $\mathcal V$. A 2-representation $\mathcal V$ is said to be \textit{integrable} if for every $M \in \mathcal V$, we have $E^i(M)=F^i(M)=0$ for some $i \in \mathbb{N}$.\\
	
	Given 2-representations $\mathcal V,\mathcal W$, a \textit{morphism of 2-representations} $D : \mathcal V \rightarrow \mathcal W$ is the data of:
	\begin{itemize}
		\item functors $D : \mathcal V_{\lambda} \rightarrow \mathcal W_{\lambda}$ for all $\lambda \in \mathbb{Z}$,
		\item natural isomorphisms $\alpha : DE \xrightarrow{\sim} ED$ and $\beta : DF \xrightarrow{\sim} FD$,
	\end{itemize}
	such that $\alpha$ and $\beta$ are compatible with the 2-morphisms $x,\tau,\varepsilon$ and $\eta$ of $\U$. This compatibility condition means that the following diagrams commute:
	\begin{align*}
		&\hspace{0.54cm} \xymatrix{
			DF \ar[r]^-{Dx} \ar[d]_-{\beta} & DF \ar[d]^-{\beta} \\
			FD \ar[r]_-{xD} & FD
		} \quad \quad \xymatrix{
		DF^2 \ar[r]^-{D\tau} \ar[d]_-{F\beta\circ \beta F} & DF^2 \ar[d]^-{F\beta\circ \beta F} \\
		F^2D \ar[r]_-{\tau D} & F^2D
	} \\
		&\xymatrix{
		DEF \ar[r]^-{D\varepsilon} \ar[d]_-{E\beta \circ \alpha F} & D \\
		EFD \ar[ru]_-{\varepsilon D} &
		} \hspace{1.6cm} \xymatrix{
		D \ar[r]^-{D\eta} \ar[dr]_-{\eta D} & DFE \ar[d]^-{F\alpha\circ \beta E} \\
		 & FED
		}
	\end{align*}
	Then, given $D$ a morphism of 2-representations and $C$ a complex of 1-morphisms of $\U$, we have a canonical isomorphism of complexes of functors $DC \simeq CD$.\\
	
	The following computational result will be useful. It is a slightly different presentation of \cite[Proposition 9.8]{lauda}.
	
	\begin{prop}\label{adj}
		Let $\mathcal V$ be a 2-representation of $\mathcal U$ and denote by $\mathcal W$ the category of endofunctors of $\mathcal V$. Let $\lambda \in \mathbb{Z}$ and denote by $\Phi : U1_{\lambda} \rightarrow \mathcal{W}$ the functor induced by the structure of 2-representation on $\mathcal V$. Let $a,b,c,d \in \mathbb{N}$ be such that $a-b=c-d$. If $\lambda \geqslant b-a$, there is an isomorphism
		\begin{multline*}
			 \mathrm{Hom}^{\bullet}_{\mathcal W}\left(\Phi\left(F^{(b)}E^{(a)}1_{\lambda}\right), \Phi\left(F^{(d)}E^{(c)}1_{\lambda}\right)\right) \simeq \\
			 \bigoplus_{i=0}^{\mathrm{min} \left \{a,c\right \}} q^{(a+c-i)(\lambda+a+c-i)}\qbin{\lambda+a+c}{i}\qbin{b+c-i}{b}\qbin{a+d-i}{d} \mathrm{End}^{\bullet}_{\mathcal W}\left(\Phi\left(F^{(a+d-i)}1_{\lambda+2(a+c-i)}\right)\right).
		\end{multline*}		
		If $\lambda \leqslant b-a$, there is an isomorphism
		\begin{multline*}
			\mathrm{Hom}^{\bullet}_{W}\left(\Phi\left(E^{(a)}F^{(b)}1_{\lambda}\right),\Phi\left( E^{(c)}F^{(d)}1_{\lambda}\right)\right) \simeq \\
			\bigoplus_{i=0}^{\mathrm{min}\left \{b,d \right\}} q^{(b+d-i)(b+d-i-\lambda)} \qbin{b+d-\lambda}{i}\qbin{a+d-i}{a}\qbin{b+c-i}{c}\mathrm{End}^{\bullet}_{\mathcal W}\left(\Phi\left(E^{(a+d-i)}1_{\lambda-2(b+d-i)}\right)\right).
		\end{multline*}
	\end{prop}

	\begin{proof}
		We prove the first isomorphism, the second one being proved similarly. There is an adjoint pair $\left(q^{a(\lambda+a)}F^{(a)}1_{\lambda+2a},E^{(a)}1_{\lambda}\right)$ providing an isomorphism
		\[
			\mathrm{Hom}^{\bullet}_{\mathcal W}\left(F^{(b)}E^{(a)}1_{\lambda}, F^{(d)}E^{(c)}1_{\lambda}\right) \simeq q^{a(\lambda+a)}\mathrm{Hom}^{\bullet}_{\mathcal W}\left(F^{(b)}1_{\lambda+2a},F^{(d)}E^{(c)}F^{(a)}1_{\lambda+2a}\right).
		\]
		Assume $\lambda \geqslant b-a=d-c$. In particular, we have $\lambda+a+c \geqslant 0$. Thus by (\ref{decomp}) we have an isomorphism
		\[
			E^{(c)}F^{(a)}1_{\lambda+2a} \simeq \bigoplus_{i=0}^{\mathrm{min}\left\{a,c\right\}} \qbin{\lambda+a+c}{i} F^{(a-i)}E^{(c-i)}1_{\lambda+2a}.
		\]
		Using the adjoint pair $\left(E^{(c-i)}1_{\lambda+2a},q^{-(c-i)(\lambda+2a+c-i)}F^{(c-i)}1_{\lambda+2(a+c-i)}\right)$, we deduce that there is an isomorphism
		\begin{multline*}
			\mathrm{Hom}^{\bullet}_{\mathcal W}\left(F^{(b)}E^{(a)}1_{\lambda}, F^{(d)}E^{(c)}1_{\lambda}\right) \simeq \\ \bigoplus_{i=0}^{\mathrm{min}\left\{a,c\right\}} q^{(a+c-i)(\lambda+a+c-i)} \qbin{\lambda+a+c}{i} \mathrm{Hom}^{\bullet}_{\mathcal W}\left(F^{(b)}F^{(c-i)}1_{\lambda+2(a+c-i)}, F^{(d)}F^{(a-i)}1_{\lambda+2(a+c-i)}\right).
		\end{multline*}
		Using the isomorphisms (\ref{decompdp}), we conclude that
		\begin{multline*}
			\mathrm{Hom}^{\bullet}_{\mathcal W}\left(F^{(b)}E^{(a)}1_{\lambda}, F^{(d)}E^{(c)}1_{\lambda}\right) \simeq \\  \bigoplus_{i=0}^{\mathrm{min} \left \{a,c\right \}} q^{(a+c-i)(\lambda+a+c-i)}\qbin{\lambda+a+c}{i}\qbin{b+c-i}{b}\qbin{a+d-i}{d} \mathrm{End}^{\bullet}_{\mathcal W}\left(F^{(a+d-i)}1_{\lambda+2(a+c-i)}\right).
		\end{multline*}
	\end{proof}

	\bigskip

	\section{Faithfulness of simple 2-representations}
	
	\subsection{Simple 2-representations} We now define the simple 2-representations of $\mathcal U$, following \cite{2km}. For $n \in \mathbb N$, and $k \in \lbrace 0,\ldots,n\rbrace$, we denote by $H_{k,n}$ the subalgebra of $H_n$ generated by $P_n^{\s_n}$, $\tau_1,\ldots,\tau_{k-1}$ and $x_1,\ldots,x_k$. We have an isomorphism of algebras $H_{k,n} \simeq H_k \otimes P_{n-k}^{\s_{n-k}}$, and a tower of algebras
	\[
		P_n^{\s_n} = H_{0,n} \subseteq H_{1,n} \subseteq \ldots \subseteq H_{n,n}=H_n.
	\]	
	The 2-representation $\cL(n)$ of $\mathcal U$ is defined as follows:
	\begin{itemize}
		\item for $k \in \lbrace0,\ldots,n\rbrace$, we define $\mathcal{L}(n)_{-n+2k} = H_{k,n}\mathrm{-proj}$, the category of finitely generated, projective and graded $H_{k,n}$-modules,
		\item the functor $E : \cL(n)_{-n+2k} \rightarrow \cL(n)_{-n+2(k+1)}$ is $\mathrm{ind}_{H_{k,n}}^{H_{k+1,n}}$ and the functor $F :\cL(n)_{-n+2k} \rightarrow \cL(n)_{-n+2(k-1)}$ is $q^{2k-n-1}\mathrm{res}_{H_{k-1,n}}^{H_{k,n}}$,
		\item the map $x : F \rightarrow F$ is left multiplication by $x_{k}$ on $\mathrm{res}_{H_{k-1,n}}^{H_{k,n}}$ and the map $\tau : F^2 \rightarrow F^2$ is left multiplication by $\tau_{k-1}$ on $\mathrm{res}_{H_{k-2,n}}^{H_{k,n}}$,
		\item the maps $\varepsilon$ and $\eta$ are the counit and unit of the canonical adjunction between induction and restriction.
	\end{itemize}

	\begin{prop}[\cite{2km}]
		The above data defines a 2-representation of $\mathcal U$ on $\cL(n)$.
	\end{prop}

	Define
	\[
		\cL(n)\mathrm{-bim} = \bigoplus_{0\leqslant k,\ell \leqslant n} (H_{k,n},H_{\ell,n})\mathrm{-bim}
	\]
	where $(H_{k,n},H_{\ell,n})\mathrm{-bim}$ is the category of finitely generated graded $(H_{k,n},H_{\ell,n})$-bimodules. The structure of 2-representation of $\U$ on $\cL(n)$ induces functors $\Phi_n : \U1_\lambda \rightarrow \cL(n)\mathrm{-bim}$ for all $\lambda \in \mathbb{Z}$. For $k \in \lbrace 0,\ldots,n\rbrace$ and $a \in \lbrace 0,\ldots,k \rbrace$ we have
	\begin{align*}
		& \Phi_n\Big( F^{a}1_{-n+2k} \Big) = q^{a(2k-n-a)} H_{k,n} \, \text{ as } (H_{k-a,n},H_{k,n})\text{-bimodules},\\
		& \Phi_n\left(E^{a}1_{-n+2(k-a)}\right) = H_{k,n} \, \text{ as } (H_{k,n},H_{k-a,n})\text{-bimodules}. \\
	\end{align*}
	Let us now describe explicitly the images of the divided powers under $\Phi_n$. Given two integers $r<\ell \in \lbrace 1,\ldots,n\rbrace$, we put
	\begin{align*}
		& x_{[r,\ell]} = x_{r+1}x_{r+2}^2\ldots x_{\ell}^{\ell-r} \in P_{n}, \\
		& x'_{[r,\ell]}=(-1)^{\frac{(\ell-r)(\ell-r-1)}{2}}x_r^{\ell-r} x_{r+1}^{\ell-r-1}\ldots x_{\ell-1} \in P_{n},\\
		& e_{[r,\ell]} = x_{[r,\ell]}\tau_{\omega_0[r,\ell]}\in H_{\ell,n}, \\
		& e'_{[r,\ell]} = \tau_{\omega_0[r,\ell]}x'_{[r,\ell]} \in H_{\ell,n}.
	\end{align*}
	Note that $e_{[r,\ell]}$ and $e'_{[r,\ell]}$ are orthogonal idempotents of $H_{\ell,n}$. Then we have
	\begin{align*}
		& \Phi_n\left(F^{(a)}1_{-n+2k}\right) = q^{a(2k-n-a) -\frac{a(a-1)}{2}}e_{[k-a+1,k]}H_{k,n} \, \text{ as } (H_{k-a,n},H_{k,n})\text{-bimodules},\\
		& \Phi_n\left(E^{(a)}1_{-n+2(k-a)}\right) = q^{-\frac{a(a-1)}{2}}H_{k,n}e'_{[k-a+1,k]} \, \text{ as } (H_{k,n},H_{k-a,n})\text{-bimodules}.
	\end{align*}
	
	A crucial fact about the simple 2-representations is the following universal property.
	
	\begin{thm}\cite[Proposition 5.15]{2km}\label{univ}
		Let $\mathcal V$ be an integrable 2-representation, and let $M \in \mathcal V_{n}$ be such that $E(M)=0$. Then there is a unique morphism of 2-representations $R_M : \cL(n) \rightarrow \mathcal V$ such that $R_M(P_{n})=M$, where $P_n$ is seen as an object of $\cL(n)_n$, with the action of $H_n$ being the polynomial representation.
	\end{thm}
	
	\medskip

	\subsection{Faithfulness} 
	The goal of this subsection is to prove the following faithfulness result.
	
	\begin{thm}\label{lifthom}
		\begin{enumerate}
			\item\label{liftnh} Let $C$ be a complex of 1-morphisms of $\U$. If $\Phi_n(C)$ is null-homotopic for all $n \in \mathbb{N}$, then $C$ is null-homotopic.
			\item\label{lifthe} Let $f$ be a morphism between complexes of 1-morphisms of $\U$. If $\Phi_n(f)$ is a homotopy equivalence for all $n \in \mathbb{N}$, then $f$ is a homotopy equivalence.
		\end{enumerate}
	\end{thm}

	\subsubsection{Generalities about Krull-Schmidt categories}  Theorem \ref{lifthom} will be a consequence of the following general result.
	
	\begin{thm}\label{liftnhgen}
		Let $\mathcal C, (\mathcal D_i)_{i \in I}$ be Krull-Schmidt categories, and let $\left(\Phi_i : \mathcal C \rightarrow \mathcal D_i\right)_{i \in I}$ be linear functors. Assume that for any finite collection $X_1,\ldots,X_n$ of indecomposable and pairwise non-isomorphic objects of $\mathcal C$, there exists $i \in I$ such that:
		\begin{enumerate}
			\item  $\Phi_i(X_1),\ldots,\Phi_i(X_n)$ are indecomposable and pairwise non-isomorphic,
			\item for all $j \in \lbrace 1,\ldots,n\rbrace$, the morphism of $K$-algebras $\mathrm{End}_{\mathcal C}(X_j) \rightarrow \mathrm{End}_{\mathcal D_i}(\Phi_i(X_j))$ induced by $\Phi_i$ is local.
		\end{enumerate}
		Let $C \in \mathrm{Comp}(\mathcal C)$. Then $C$ is null-homotopic if and only if for all $i \in I$, $\Phi_i(C)$ is null-homotopic.
	\end{thm}

	Let us introduce some terminology that we will use throughout the proof. With the notations of Theorem \ref{liftnhgen}, given a finite collection $X_1,\ldots,X_n$ of indecomposable and pairwise non-isomorphic objects of $\mathcal C$, we say that an element $i \in I$ is \textit{adapted} to $X_1,\ldots,X_n$ if conditions (1) and (2) of Theorem \ref{liftnhgen} are satisfied. Similarly, given an object $M$ of $\mathcal C$, we say that $i \in I$ is adapted to $M$ if it is adapted to the indecomposable summands of $M$. Given an arrow $f : M \rightarrow N$ of $\mathcal C$, we say that $i \in I$ is adapted to $f$ if it is adapted to $M\oplus N$.\\
	
	The first step to prove Theorem \ref{liftnhgen} is to prove a lifting result for split injections and split surjections. This will be based on the following Lemma.

	\begin{lem}\label{ks}
		Let $\mathcal C$ be a Krull-Schmidt category, and let $f : M \rightarrow N$, $g : N \rightarrow L$ be arrows in $\mathcal C$. Assume given a decomposition $N = \oplus_j N_j$ in indecomposable objects. Denote by $i_j : N_j \rightarrow N, p_j : N \rightarrow N_j$ the inclusion and projection morphisms associated to this decomposition.
		\begin{enumerate}
			\item Assume that $L$ is indecomposable and that $gf$ is a split surjection. Then there exists $j$ such that $gi_j$ is an isomorphism and $p_jf$ is a split surjection.
			\item Assume that $M$ is indecomposable and that $gf$ is a split injection. Then there exists $j$ such that $gi_j$ is a split injection and $p_jf$ is an isomorphism.
		\end{enumerate}
	\end{lem}

	\begin{proof}
		We prove the first statement, the second one being similar. Fix a right inverse $u$ of $gf$. Then we have:
		\[
			\sum_j gi_jp_jfu = \mathrm{id}_L.
		\]
		Since $\mathrm{End}_{\mathcal C}(L)$ is a local $K$-algebra, there exists $j$ such that $gi_jp_jfu$ is invertible. Thus $gi_j$ is a split surjection. Since $N_j$ is indecomposable, it follows that $gi_j$ is an isomorphism. Hence $p_jfu$ is invertible, so $p_jf$ is a split surjection.
	\end{proof}

	\begin{prop}\label{liftss}
		Let $\mathcal C, \mathcal D$ be Krull-Schmidt categories, and let $\Phi : \mathcal C \rightarrow \mathcal D$ be a functor such that:
		\begin{enumerate}
			\item for any indecomposable object $X \in \mathcal C$, $\Phi(X)$ is indecomposable and the morphism of $K$-algebras $\mathrm{End}_{\mathcal C}(X) \rightarrow \mathrm{End}_{\mathcal D}(\Phi(X))$ induced by $\Phi$ is local,
			\item for any indecomposable objects $X,Y \in \mathcal{C}$, $X \simeq Y$ if and only if $\Phi(X) \simeq \Phi(Y)$.
		\end{enumerate}
		Let $f : M \rightarrow N$ be a morphism in $\mathcal C$. Then $f$ is a split surjection (resp. injection) if and only if $\Phi(f)$ is a split surjection (resp. injection).
	\end{prop}
	
	\begin{proof}
		We treat the case of split surjections, the case of split injections being similar. If $f$ is a split surjection, it is clear that $\Phi(f)$ is a split surjection. Conversely, assume that $\Phi(f)$ is a split surjection. We proceed by induction on the number of isomorphism classes of indecomposable summands of $N$. The result is clear if $N=0$.
		
		Assume that $N \neq 0$ and fix a decomposition $N=N_1\oplus N'$ with $N_1$ indecomposable. Denote by $p : N \rightarrow N_1$ the projection associated with this decomposition. We also fix a decomposition $M = \oplus_{j=1}^m M_j$ in indecomposable objects, and we denote by $i_j : M_j \rightarrow M, p_j : M \rightarrow M_j$ the inclusion and projection morphisms associated with this decomposition.
		
		The composition $\Phi(M) \xrightarrow{\mathrm{id}} \Phi(M) \xrightarrow{\Phi(pf)} \Phi(N_1)$ is a split surjection and $\Phi(N_1), \left(\Phi(M_j)\right)_j$ are indecomposable by assumption (1). Thus by Lemma \ref{ks}, there exists $j$ such that $\Phi(pfi_j)$ is an isomorphism. Without loss of generality, we can assume that $j=1$. Then $\Phi(N_1)$ and $\Phi(M_1)$ are isomorphic, so by assumption (2) $N_1$ and $M_1$ are isomorphic. Fix an isomorphism $g : N_1 \xrightarrow{\sim} M_1$. The morphism $\mathrm{End}_{\mathcal C}(N_1) \rightarrow \mathrm{End}_{\mathcal D}(\Phi(N_1))$ induced by $\Phi$ is local and $\Phi(pfi_1g)$ is invertible. Thus $pfi_1g$ is invertible, and $pfi_1$ is invertible as well.
		
		Hence, if we let $M' = \oplus_{j >1} M_j$, in the decompositions $M = M_1 \oplus M'$ and $N = N_1 \oplus N'$ we can write $f$ as a matrix
		\[
			f = \left(\begin{array}{cc}
			a & b \\
			c & d
			\end{array}\right)
		\]
		with $a = pfi_1$. Since $a$ is an isomorphism, there exist automorphisms $u$ of $N$ and $v$ of $M$ such that
		\[
			ufv = \left(\begin{array}{cc}
			a & 0 \\
			0 & f'
			\end{array}\right)
		\]
		for some morphism $f' : M' \rightarrow N'$. Since $\Phi(f)$ is a split surjection and $u,v$ are invertible, $\Phi(ufv)$ is a split surjection as well. Thus, $\Phi(f')$ is a split surjection. The number of indecomposable summands of $N'$ is one less than that of $N$. By induction, we deduce that $f'$ is a split surjection. It follows that $ufv$ is a split surjection. So $f$ is a split surjection.
	\end{proof}

	\begin{cor}\label{liftssi}
		Let $\mathcal C, (\mathcal D_i)_{i \in I}$ be Krull-Schmidt categories, and let $\left(\Phi_i : \mathcal C \rightarrow \mathcal D_i\right)_{i \in I}$ be functors satisfying the same assumptions as in Theorem \ref{liftnh}. Let $f : M \rightarrow N$ be a morphism in $\mathcal C$. The following are equivalent:
		\begin{enumerate}
			\item $f$ is a split surjection (resp. injection).
			\item there exists $i \in I$ adapted to $f$ such that $\Phi_i(f)$ is a split surjection (resp. injection).
			\item for all $i \in I$, $\Phi_i(f)$ is a split surjection (resp. injection).
		\end{enumerate}
	\end{cor}
	
	\begin{proof}
		Denote by $\mathcal C_f$ the smallest full subcategory of $\mathcal C$ containing $M \oplus N$ and closed under direct sums and direct summands. This is a Krull-Schmidt category with finitely many isomorphism classes of indecomposable objects. By assumption, we can fix $i \in I$ adapted to $f$. Then the functor $\mathcal C_f \rightarrow \mathcal{D}_i$ induced by $\Phi_i$ satisfies the assumptions of Proposition \ref{liftss}. The result then follows from Propostion \ref{liftss}.
	\end{proof}

	We can now prove Theorem \ref{liftnhgen} for complexes that are bounded above or below.

	\begin{prop}\label{bounded}
		Let $\mathcal C, (\mathcal D_i)_{i \in I}$ be Krull-Schmidt categories, and let $\left(\Phi_i : \mathcal C \rightarrow \mathcal D_i\right)_{i\in I}$ be functors satisfying the same assumptions as in Theorem \ref{liftnh}. Let $C \in \mathrm{Comp}(\mathcal C)$ be bounded above or below. Then C is null-homotopic if and only if for all $i \in I$, $\Phi_i(C)$ is null-homotopic.
	\end{prop}
	
	\begin{proof}
		We treat the case where $C$ is bounded above, the bounded below case being similar. If $C$ is null-homotopic, it is clear that for all $i \in I$, $\Phi_i(C)$ is null-homotopic. Conversely, assume that for all $i \in I$, $\Phi_i(C)$ is null-homotopic. Since $C$ is bounded above, it has the form
		\[
		\cdots \rightarrow C^j \xrightarrow{d^j} C^{j+1} \rightarrow \cdots \rightarrow C^{\ell-1} \xrightarrow{d^{\ell-1}} C^{\ell} \rightarrow 0,
		\]
		with $\ell$ such that $C^{j}=0$ for all $j>\ell$. We inductively construct maps $h_j \in \mathrm{Hom}_{\mathcal C}(C^j,C^{j-1})$ for $j \leqslant \ell$, such that  $h_{j+1}d^j+d^{j-1}h_j=\mathrm{id}_{C^j}$.
		
		By assumption, $\Phi_i(d^{\ell-1})$ is a split surjection for all $i\in I$. Hence $d^{\ell -1}$ is a split surjection by Corollary \ref{liftssi}. We define $h_{\ell}$ to be a right inverse to $d^{\ell-1}$. Assume that the $h_i$ are constructed for $i\geqslant j$. From $h_{j+1}d^j+d^{j-1}h_j=\mathrm{id}_{C^j}$ and $d^jd^{j-1}=0$, we get that $e= h_jd^{j-1}$ is an idempotent of $C^{j-1}$. Hence there is a decomposition $C^{j-1}=\mathrm{im}(e)\oplus\mathrm{im}(1-e)$. We have $ed^{j-2} =0$ and $d^{j-1}(1-e) = 0$. Thus $C$ has the form:
		\[
		\cdots \rightarrow C^{j-2} \xrightarrow{d^{j-2} = \left(\begin{smallmatrix} 0 \\ \delta \end{smallmatrix}\right)} \mathrm{im}(e) \oplus \mathrm{im}(1-e) \xrightarrow{d^{j-1} = \left( \begin{smallmatrix} \star & 0 \end{smallmatrix}\right)} C^j \rightarrow \cdots
		\]
		for some morphism $\delta : C^{j-2} \rightarrow \mathrm{im}(1-e)$. For all $i\in I$, $\Phi_i(\delta)$ is a split surjection since $\Phi_i(C)$ is null-homotopic. Thus, $\delta$ is also a split surjection by Corollary \ref{liftssi}. Let $\delta'$ be a right inverse to $\delta$, and define $h_{j-1} = \left( \begin{smallmatrix} 0 & \delta' \end{smallmatrix}\right) : \mathrm{im}(e) \oplus \mathrm{im}(1-e) \rightarrow C^{j-2}$. By construction, we have $h_jd^{j-1}+d^{j-2}h_{j-1}=\mathrm{id}_{C^{j-1}}$, which completes the induction.
	\end{proof}

	To finish the proof of Theorem \ref{liftnhgen}, we show how to extend this result to unbounded complexes. We will need the following standard Lemma.

	\begin{lem}[Gaussian elimination]\label{gauss}
		\begin{enumerate}
			\item Let $\mathcal C$ be an additive category, and let $C \in \mathrm{Comp}(\mathcal C)$. Assume that $C$ has the form:
			\[
			C = \cdots \rightarrow C^{i-1} \rightarrow X \oplus Y \xrightarrow{\left(\begin{smallmatrix} a & b \\ c & d \end{smallmatrix}\right)} Z \oplus W \rightarrow C^{i+2} \rightarrow \cdots
			\]
			with $a$ an isomorphism. Then $C$ is homotopy equivalent to a complex of the form:
			\[
			\cdots \rightarrow C^{i-1} \rightarrow Y \rightarrow W \rightarrow C^{i+2} \rightarrow \cdots
			\]
			\item Let $\mathcal C$ be an additive, idempotent complete category, and let $C \in \mathrm{Comp}(\mathcal C)$. Assume that $C$ has the form:
			\[
			C = \cdots \rightarrow C^{i-1} \rightarrow C^i \xrightarrow{\left(\begin{smallmatrix} a  \\ c  \end{smallmatrix}\right)} Z \oplus W \rightarrow C^{i+2} \rightarrow \cdots
			\]
			with $a$ a split surjection. Then $C$ is homotopy equivalent to a complex of the form
			\[
			\cdots \rightarrow C^{i-1} \rightarrow Y \rightarrow W \rightarrow C^{i+2} \rightarrow \cdots
			\]
			where $Y$ is an object of $\mathcal C$ such that $C^i \simeq Y \oplus Z$.
			\item Let $\mathcal C$ be an additive, idempotent complete category, and let $C \in \mathrm{Comp}(\mathcal C)$. Assume that $C$ has the form:
			\[
			C = \cdots \rightarrow C^{i-1} \rightarrow X \oplus Y \xrightarrow{\left(\begin{smallmatrix} a  & b  \end{smallmatrix}\right)} C^{i+1} \rightarrow C^{i+2} \rightarrow \cdots
			\]
			with $a$ a split injection. Then $C$ is homotopy equivalent to a complex of the form
			\[
			\cdots \rightarrow C^{i-1} \rightarrow Y \rightarrow W \rightarrow C^{i+2} \rightarrow \cdots
			\]
			where $W$ is an object of $\mathcal C$ such that $C^{i+1} \simeq X \oplus W$.
		\end{enumerate}
	\end{lem}
	
	\begin{proof}
		\begin{enumerate}
			\item Since $a$ is an isomorphism, there exists an automorphism $p$ of $X \oplus Y$, an automorphim $q$ of $Y \oplus Z$ and a morphism $d' : Y \rightarrow W$ such that
			\[
			q\left(\begin{smallmatrix} a & b \\ c & d \end{smallmatrix}\right)p = \left(\begin{smallmatrix} a & 0 \\ 0 & d' \end{smallmatrix}\right).
			\]
			Hence $C$ is isomorphic to
			\[
			\cdots \rightarrow C^{i-1} \rightarrow X \oplus Y \xrightarrow{\left(\begin{smallmatrix} a & 0 \\ 0 & d' \end{smallmatrix}\right)} Z \oplus W \rightarrow C^{i+2} \rightarrow \cdots
			\]
			Since $a$ is a isomorphism, the compositions $C^{i-1} \rightarrow X\oplus Y \twoheadrightarrow X$ and $X \hookrightarrow X\oplus  Y \rightarrow C^{i+1}$ are zero. Hence $C$ is isomorphic to a direct sum
			\[
			\left(\cdots \rightarrow C^{i-1} \rightarrow Y \xrightarrow{d'} W \rightarrow C^{i+2} \rightarrow \cdots \right) \bigoplus \left( 0 \rightarrow X \xrightarrow{a} Z \rightarrow 0\right).
			\]
			The complex $0 \rightarrow X \xrightarrow{a} Z \rightarrow 0$ is homotopy equivalent to zero since $a$ is an isomorphism, so the result follows.
			\item Fix a right inverse $a'$ of $a$. Then $a'a$ is an idempotent of $\mathrm{End}_{\mathcal C}(C^i)$. Put $X = \mathrm{im}(a'a)$ and $Y = \mathrm{im}(1-a'a)$. We have a decomposition $C^i = X \oplus Y$ and $a$ induces an isomorphism $X \xrightarrow{\sim} Z$. The result then follows from (1).
			\item The proof is similar to (2).
		\end{enumerate}
	\end{proof}

	\begin{proof}[Proof of Theorem \ref*{liftnhgen}]
		For $C \in \mathrm{Comp}(\mathcal C)$, we denote by $n(C)$ the number of summands in a decomposition of $C^0$ into indecomposable components. Assume that for all $i \in I$, $\Phi_i(C)$ is null-homotopic. We start by proving that $d_C^0=0$ or $C$ is homotopy equivalent to a complex $C'$ with $n(C') = n(C) -1$.
	
		Fix $i \in I$ adapted to $d_C^{-1}$ and $d_C^0$. If $\Phi_i(d_C^0) = 0$, then $\Phi_i(d_C^{-1})$ is a split surjection since $\Phi_i(C)$ is null-homotopic. By Corollary \ref{liftssi}, we deduce that $d_C^{-1}$ is a split surjection. Thus, $d_C^0 = 0$. Assume now that $\Phi_i(d_C^{0}) \neq 0$. Fix a decomposition into indecomposable summands $C^0 = \oplus_{j} M_j$. Since $\Phi_i(C)$ is null-homotopic, there exists a split injection $u : M \rightarrow \Phi_i(C^0)$ with $M \in \mathcal D_i$ indecomposable such that $\Phi_i(d_C^0)u$ is a split injection. By Lemma \ref{ks}, there exists $j$ such that $\Phi_i(d^0_Ci_j)$ is a split injection, where $i_j : M_j \rightarrow C^0$ is the inclusion associated with the decomposition of $C^0$. By Corollary \ref{liftssi}, we deduce that $d_C^0u_j$ is a split injection. Thus, by Lemma \ref{gauss}, $C$ is homotopy equivalent to a complex $C'$ with $n(C') = n(C) -1$. 
	
		Thus, if $\Phi_i(C)$ is null-homotopic for all $i \in I$, we have proved that the $d_C^0=0$, or $C$ is homotopy equivalent to a complex $C'$ with $n(C')=n(C)-1$. By induction, it follows that $C$ is homotopy equivalent to a complex $C'$ such that $d_{C'}^0=0$. Hence, $C$ is homotopy equivalent to a direct sum $C_1 \oplus C_2$, with $C_1$ bounded above and $C_2$ bounded below. For all $i \in I$, $\Phi_i(C_1)$ and $\Phi_i(C_2)$ are null-homotopic since $\Phi_i(C)$ is as well. So by Proposition \ref{bounded}, $C_1$ and $C_2$ are null-homotopic. Hence, $C$ is null-homotopic. Conversely, if $C$ is null-homotopic, it is clear that for all $i \in I$, $\Phi_i(C)$ is null-homotopic.
	\end{proof}

	\begin{cor}\label{lifthegen}
		Let $\mathcal C, (\mathcal D_i)_{i \in I}$ be Krull-Schmidt categories, and let $\left(\Phi_i : \mathcal C \rightarrow \mathcal D_i\right)_{i \in I}$ be functors satisfying the same assumptions as in Theorem \ref{liftnhgen}. Let $f : C \rightarrow D$ be a morphism between complexes of objects of $\mathcal C$. Then $f$ is a homotopy equivalence if and only if for all $i \in I$, $\Phi_i(f)$ is a homotopy equivalence.
	\end{cor}
	
	\begin{proof}
		It suffices to apply Theorem \ref{liftnhgen} to the complex $\mathrm{Cone}(f)$.
	\end{proof}

	\subsubsection{Application to simple 2-representations} To prove our faithfulness result, we show that the collection of simple 2-representations of $\U$ satisfies the assumptions of Theorem \ref{liftnhgen}. The first step is to study the image of the indecomposable 1-morphisms of $\U$ in the simple 2-representations. This will be based on the following observation.

	\begin{lem}\label{polsym}
		Let $k,\ell,r \in \mathbb{N}$ and put $n=k+\ell+r$. Then the multiplication map
		\[
			P_n^{\s_{k+\ell}\times\s_r}\otimes P_n^{\s_k\times\s_{\ell+r}} \rightarrow P_n^{\s_k\times\s_{\ell}\times\s_r}
		\]
		is surjective.
	\end{lem}

	\begin{proof}
		Let $I$ be the image of the multiplication map. We denote by $\epsilon_j$ the elementary symmetric polynomial of degree $j$. To prove the result, it suffices to show that $\epsilon_j(x_{k+1},\ldots,x_{k+\ell}) \in I$ for all $j \geqslant 1$. For $j=1$, we have
		\[
			\epsilon_1(x_{k+1},\ldots,x_{k+\ell}) = \epsilon_1(x_{1},\ldots,x_{k+\ell}) - \epsilon_1(x_{1},\ldots,x_{k}) \in I.
		\]
		In general, we have:
		\[
			\epsilon_{j+1}(x_{k+1}\ldots,x_{k+\ell}) = \epsilon_{j+1}(x_1\ldots,x_{k+\ell}) - \sum_{i=1}^{j+1} \epsilon_i(x_1,\ldots,x_k)\epsilon_{j+1-i}(x_{k+1},\ldots,x_{k+\ell}),
		\]
		and the result follows by induction on $j$.
	\end{proof}

	\begin{prop}\label{indec}
		Let $X$ be an indecomposable object of $\mathcal U1_{\lambda}$ for some $\lambda \in \mathbb{Z}$. If $n$ is large enough and of the same parity as $\lambda$, then $\left| \Phi_n(X) \right| \in 1 + q\mathbb{N}[q]$.
	\end{prop}

	\begin{proof}
		We start by proving the result for $X=E^{(a)}1_{\lambda}$ for some $a\in \mathbb{N}$. Let $n>\vert \lambda \vert + 2a$ be an integer of the same parity as $\lambda$, and let $k=\frac{\lambda+n}{2}$. Then $\Phi_n\left(E^{(a)}1_{\lambda}\right) = q^{-\frac{a(a-1)}{2}} H_{k+a,n}e'_{[k+1,k+a]}$ as $(H_{k+a,n},H_{k,n})$-bimodules. By Morita equivalence, we have an isomorphism of graded $K$-algebras
		\[
			\mathrm{End}^{\bullet}_{\left(H_{k+a,n},H_{k,n}\right)}\left(H_{k+a,n}e'_{[k+1,k+a]}\right) \simeq \mathrm{End}^{\bullet}_{\left(P_n^{\s_{k+a}\times\s_{n-k-a}},P_n^{\s_k\times\s_{n-k}}\right)}\left(e_{[1,k+a]}H_{k+a,n}e'_{[1,k]}e'_{[k+1,k+a]}\right).
		\]
		There is an isomorphism of $\left(P_n^{\s_{k+a}\times\s_{n-k-a}},P_n^{\s_k\times\s_{n-k}}\right)$-bimodules
		\[
				\left \{ \begin{array}{rcl}
				P_n^{\s_k\times\s_a\times\s_{n-k-a}} & \simeq & e_{[1,k+a]}H_{k+a,n}e'_{[1,k]}e'_{[k+1,k+a]}, \\
				P & \mapsto & e_{[1,k+a]}Px'_{[1,k]}x'_{[k+1,k+a]}.
				\end{array} \right. 
		\]
		By Lemma \ref{polsym}, $P_n^{\s_k\times\s_a\times\s_{n-k-a}}$ is cyclic generated by 1 as a $\left(P_n^{\s_{k+a}\times\s_{n-k-a}},P_n^{\s_k\times\s_{n-k}}\right)$-bimodule. It follows that
		\[
			\mathrm{End}^{\bullet}_{\cL(n)\mathrm{-bim}}\left(\Phi_n\left(E^{(a)}1_{\lambda}\right)\right) = P_{n}^{\s_{k}\times\s_{a}\times\s_{n-k-a}}.
		\]
		Thus, the result holds for $X=E^{(a)}1_{\lambda}$. The case $X=F^{(a)}1_{\lambda}$ is similar.
		
		Assume now that $X=q^sF^{(b)}E^{(a)}1_{\lambda}$ for $s \in \mathbb Z$ and $\lambda \geqslant b-a$. By Proposition \ref{adj}, we have
		\[
			\left|\Phi_n\left(F^{(b)}E^{(a)}1_{\lambda}\right)\right| = \sum_{i=0}^{a} q^{(2a-i)(\lambda+2a-i)}\qbin{\lambda+2a}{i}\qbin{a+b-i}{b}^2 \left|\Phi_n\left(F^{(a+b-i)}1_{\lambda+2(2a-i)}\right)\right|.
		\]
		Let us find the lowest degree term in that expression. The lowest degree term of $\qbin{\lambda+2a}{i}$ is $q^{-i(\lambda+2a-i)}$, and that of $\qbin{a+b-i}{b}^2$ is $q^{-2b(a-i)}$. Hence, the lowest degree term of $q^{(2a-i)(\lambda+2a-i)}\qbin{\lambda+2a}{i}\qbin{a+b-i}{b}^2$ is $q^{2(a-i)(\lambda+2a-b-i)}$. For $i=a$, this term is 1. For $i<a$, we have $(a-i)(\lambda+2a-b-i)>0$ since $\lambda \geqslant b-a$. Now using the first part of the proof, we know that for $n$ large enough and of the same parity as $\lambda$, the lowest degree term of $\left|\Phi_n\left(F^{(a+b-i)}1_{\lambda+2(2a-i)}\right)\right|$ is 1 for all $i \in \lbrace 0,\ldots,a\rbrace$. Putting all this together, we indeed have
		\[
			\left|\Phi_n\left(F^{(b)}E^{(a)}1_{\lambda}\right)\right| \in 1+q\mathbb{N}[q],
		\]
		when $n$ is large enough and of the same parity as $\lambda$. The case $X=q^sE^{(a)}F^{(b)}1_{\lambda}$ with $\lambda \leqslant b-a$ and $s \in \mathbb Z$ is proved similarly.
	\end{proof}

	In particular, if $X$ is an indecomposable object of $U1_{\lambda}$ and $n$ is large enough and of the same parity as $\lambda$, then $\Phi_n(X)$ is indecomposable and the morphism $\mathrm{End}_{\U}(X) \rightarrow \mathrm{End}_{\cL(n)\mathrm{-bim}}(\Phi_n(X))$ induced by $\Phi_n$ is local (both sides are equal to $K$). The next step is to check that the simple 2-representations distinguish non-isomorphic indecomposable 1-morphisms of $\U$.
	
	\begin{prop}\label{iso}
		Let $X,Y$ be indecomposable objects of $\mathcal U1_{\lambda}$ for some $\lambda \in \mathbb{Z}$. If $X$ and $Y$ are not isomorphic, then $\Phi_n(X)$ and $\Phi_n(Y)$ are not isomorphic for $n \in \mathbb N$ large enough and of the same parity as $\lambda$.
	\end{prop}

	\begin{proof}
		Assume $X=F^{(b)}E^{(a)}1_{\lambda}$ and $Y=F^{(d)}E^{(c)}1_{\lambda}$ with $\lambda \geqslant b-a=d-c$, and $a\neq c$. By Proposition \ref{adj} we have
		\[
			\left<\Phi_n(X), \Phi_n(Y)\right> = \sum_{i=0}^{\mathrm{min} \left \{a,c\right \}} q^{(a+c-i)(\lambda+a+c-i)}\qbin{\lambda+a+c}{i}\qbin{b+c-i}{b}\qbin{a+d-i}{d} \left|\Phi_n\left(F^{(a+d-i)}1_{\lambda+2(a+c-i)}\right)\right|.
		\]
		The lowest degree term of of $q^{(a+c-i)(\lambda+a+c-i)}\qbin{\lambda+a+c}{i}\qbin{b+c-i}{b}\qbin{a+d-i}{d}$ is
		\[
			q^{(a+c-i)(\lambda +a+c-i)-i(\lambda+a+c-i)-b(c-i)-d(a-i)}.
		\]
		Let us rewrite the exponent. We have:
		\begin{align*}
			& (a+c-i)(\lambda +a+c-i)-i(\lambda+a+c-i)-b(c-i)-d(a-i)\\
			& = (a-i)(\lambda +a+c-i)+(c-i)(\lambda+a+c-i)-b(c-i)-d(a-i) \\
			& = (a-i)(\lambda+a+c-d-i) + (c-i)(\lambda+a-b+c-i) \\
			& = (a-i)(\lambda+2a-b-i) + (c-i)(\lambda+2c-d-i).
		\end{align*}
		Without loss of generality, we can assume that $a>c$. Then using $i\leqslant c<a$ and $\lambda \geqslant b-a=d-c$ we see that $(a-i)(\lambda+2a-b-i)>0$ and $(c-i)(\lambda+2c-d-i)\geqslant 0$. By Proposition \ref{indec}, we have $\left|\Phi_n\left(F^{(a+d-i)}1_{\lambda+2(a+c-i)}\right)\right| \in \mathbb{N}[q]$ for all $i \in \lbrace0, \ldots, \mathrm{min}(a,c)\rbrace$ if $n$ is large enough and of the same parity as $\lambda$. Hence, $\left<\Phi_n(X), \Phi_n(Y)\right> \in q\mathbb{N}[q]$ for such an $n$. By symmetry, we also have $\left<\Phi_n(X), \Phi_n(Y)\right> \in q\mathbb{N}[q]$ if $n$ is large enough and of the same parity as $\lambda$.
		
		Thus, we can fix $n \in \mathbb{N}$ such that $\left<\Phi_n(X), \Phi_n(Y)\right>,\left<\Phi_n(Y), \Phi_n(X)\right> \in q\mathbb{N}[q]$ and $\Phi_n(X),\Phi_n(Y)$ are indecomposable. Then a composition $\Phi_n(X) \rightarrow \Phi_n(Y) \rightarrow \Phi_n(X)$ of homogeneous morphisms of arbitrary degrees is zero or of positive degree. In particular it is not invertible. Hence $\Phi_n(X)$ and $\Phi_n(Y)$ are not isomorphic. The same argument holds if $X=q^sF^{(b)}E^{(a)}1_{\lambda}$ and $Y=q^rF^{(d)}E^{(c)}1_{\lambda}$ for some $r,s \in \mathbb{Z}$, and the case of $X=q^sE^{(a)}F^{(b)}1_{\lambda}$ and $Y=q^rE^{(c)}F^{(d)}1_{\lambda}$ is proved similarly.
	\end{proof}

	\begin{proof}[Proof of Theorem \ref{lifthom}]
		Note that every complex of 1-morphisms of $\U$ can be written as a direct sum of complexes of objects of $\U1_{\lambda}$, for $\lambda$ ranging over $\mathbb{Z}$. Hence, it suffices to prove the result for complex of objects of $\U1_{\lambda}$. By Propositions \ref{indec} and \ref{iso}, the family of functors $\left(\Phi_n : \U1_{\lambda} \rightarrow \cL(n)\mathrm{-bim}\right)_{n \in \mathbb N}$ satisfies the assumptions of Theorem \ref{liftnhgen}. Thus, the result follows from Theorem \ref{liftnhgen} and Corollary \ref{lifthegen}.
	\end{proof}

	\medskip

	\subsection{Extension to the derived category}
	
	The goal of this subsection is to show that the conclusions of Theorem \ref{lifthom} still hold if the null-homotopic assumption in (\ref{liftnh}) and the homotopy equivalence assumption in (\ref{lifthe}) are weakened to acyclic and quasi-isomorphism respectively, modulo some finiteness assumptions.
	
	\begin{thm}\label{liftder}
		\begin{enumerate}
			\item Let $C$ be a complex of 1-morphisms of $\U$ such that for every integrable 2-representation $\mathcal V$ and $M \in \mathcal V$, the complex $C(M)$ is bounded. Assume that for all $n\in \mathbb{N}$, $\Phi_n(C)$ is acyclic. Then $C$ is null-homotopic.
			
			\item Let $f$ be a morphism between complexes of 1-morphisms of $\U$ that satisfy the same condition as in (1). Assume that for all $n \in \mathbb{N}$, $\Phi_n(f)$ is a quasi-isomorphism. Then $f$ is a homotopy equivalence.
		\end{enumerate}	
	\end{thm}
	
	The first step is the following result from \cite{ChR}. Since this result is not explicitly singled out in \cite{ChR}, we provide a proof. However all of the arguments are taken from \cite{ChR}, with slight modifications to change from the abelian setup of minimal 2-representations to simple 2-representations.
	
	\begin{prop}\label{nhonobj}
		Let $C$ be a complex of 1-morphisms of $\U$ such that for every integrable 2-representation $\mathcal V$ and $M \in \mathcal V$, the complex $C(M)$ is bounded. Assume that for all $n \in \mathbb{N}$, $\Phi_n(C)$ is acyclic. Then for any abelian integrable 2-representation $\mathcal V$ and $M \in \mathcal V$, the complex $C(M)$ is null-homotopic.
	\end{prop}

	\begin{proof}
		Let $\mathcal{V}$ be an abelian integrable 2-representation of $\U$. The first step is to show that for every highest weight object $N \in \mathcal V$ and $i \geqslant 0$, the complex $C(F^iN)$ is null-homotopic. Denote by $n$ the weight of $N$. By Proposition \ref{univ}, we have a morphism of 2-representations $R_N : \cL(n) \rightarrow \mathcal V$ sending $P_n$ to $N$. The complex of projective modules $C(F^iP_n)$ is bounded and acyclic by assumption. Hence $C(F^iP_n)$ is null-homotopic. It follows that $C(F^iN) \simeq R_N(C(F^iP_n))$ is null-homotopic as well.
		
		The second step is to show that for any $M\in \mathcal V$, $C(M)$ is acyclic. Let $X=C^{\vee}C(M)$, where $C^{\vee}$ denotes the right dual of $C$. Since $\mathrm{End}_{D(\mathcal V)}(C(M)) \simeq \mathrm{Hom}_{D(\mathcal V)}(M,X)$, it suffices to show that $X$ is acyclic. Assume it is not, and denote by $j$ the smallest integer such that $H^j(X) \neq 0$. Let $k$ be the maximal integer such that $E^kH^j(X)\neq 0$. Put $N=E^kH^j(X)\simeq H^j(E^kX)$, a highest weight object of $\mathcal V$. Then we have
		\[
			\mathrm{Hom}_{D(\mathcal V)}\left(F^kN,X[-j]\right) \simeq \mathrm{Hom}_{D(\mathcal V)}\left(N,E^kX[-j]\right) \neq 0,
		\]
		the isomorphism begin up to a degree shift. However, by the first step we have
		\[
			\mathrm{Hom}_{D(\mathcal V)}(F^kN,X[-j]) \simeq \mathrm{Hom}_{D(\mathcal V)}(C(F^kN),C(M)[-j])=0,
		\]
		which is a contradiction. Thus $C(M)$ is acyclic.
		
		Finally, we prove that $C(M)$ is null-homotopic. By \cite[Corollary 5.33]{ChR}, there exists an algebra $A$, a 2-representation of $\U$ on $A\mathrm{-mod}$ restricting to a 2-representation on $A\mathrm{-proj}$, and a morphism of 2-representations $S : A\mathrm{-mod} \rightarrow \mathcal{V}$ such that $M$ is a direct summand of $S(A)$. Applying the previous step to the 2-representation $A\mathrm{-mod}$ and the object $A$, we obtain that $C(A)$ is acyclic. Since it is also a bounded complex of projective $A$-modules, we conclude that $C(A)$ is null-homotopic. Hence, $C(S(A))\simeq S(C(A))$ is null-homotopic, from which we deduce that $C(M)$ is null-homotopic.
	\end{proof}

	\begin{proof}[Proof of Theorem \ref{liftder}]
		Let $n \in \mathbb N$. The 2-representation of $\U$ on $\cL(n)$ induces a 2-representation of $\U$ on $\cL(n)\mathrm{-bim}$ by tensoring on the left. Consider the object
		\[
			M_n = \bigoplus_{k=0}^n H_{k,n} \in \cL(n)\mathrm{-bim},
		\]
		where each summand $H_{k,n}$ is seen as a $(H_{k,n},H_{k,n})$-bimodule with left and right actions given by multiplication. If $C$ is a complex of 1-morphisms of $\U$ then we have $C(M_n) = \Phi_n(C)$ as objects of $\cL(n)\mathrm{-bim}$.
		
		Assume now that for every integrable 2-representation $\mathcal V$ and $M \in \mathcal V$, the complex $C(M)$ is bounded, and that $\Phi_n(C)$ is acyclic for all $n \in \mathbb{N}$. By Proposition \ref{nhonobj}, the complex $C(M_n)$ is null-homotopic. Thus $\Phi_n(C)$ is null-homotopic. The first statement then follows from Theorem \ref{lifthom}(\ref{liftnh}). We deduce the second statement by applying the first one to $C=\mathrm{Cone}(f)$. 
	\end{proof}

	\bigskip
	
	\section{Application to the Rickard complex}
	
	One of the important features of Chuang and Rouquier's approach to categorification of representations of $\mathfrak{sl}_2$ in \cite{ChR} is a categorification of the action of the simple reflection of $\mathrm{SL}_2$. This takes the form of a complex, whose definition we recall now.
	
	\begin{defi}
		Given $\lambda \in \mathbb{Z}$, we define a complex $\Theta 1_{\lambda}$ of objects of $1_{-\lambda}\mathcal{U}1_{\lambda}$ as follows.
		\begin{itemize}
			\item The $r^{\mathrm{th}}$ component of $\Theta1_{\lambda}$ is $\Theta^r1_{\lambda}=q^{-r}F^{(\lambda+r)}E^{(r)}$, where we put $F^{(\ell)}=E^{(\ell)}=0$ if $\ell<0$.
			
			\item The differential of $\Theta1_{\lambda}$ is given by the composition of $F^{(\lambda+r)}\eta E^{(r)} : F^{(\lambda+r)}E^{(r)} \rightarrow F^{(\lambda+r)}FEE^{(r)}$ with the projection on $F^{(\lambda+r+1)}E^{(r+1)}$ given by the idempotents $e_{\lambda+r+1}$ and $e'_{r+1}$. The fact that the differential squares to zero follows from $e_2e_2'=0$.
		\end{itemize}
	\end{defi}
	
	Let $n \in \mathbb{N}$. For all $\lambda \in \mathbb{Z}$, the complex $\Phi_n\left(\Theta1_{\lambda}\right)$ is bounded. Let us describe these complexes explicitly. Let $k \in \lbrace 0,\ldots,\frac{n}{2}\rbrace$. Then as complexes of graded $(H_{n-k,n},H_{k,n})$-bimodules we have
	\begin{multline}\label{2k<n}
		q^{\frac{(2k-n)(2k-n-1)}{2}}\Phi_n\left(\Theta 1_{-n+2k}\right) = \\
	 0 \rightarrow H_{n-k,n}e'_{[k+1,n-k]} \rightarrow \cdots \rightarrow e_{[n-k+1,r]}H_{r,n}e'_{[k+1,r]} \rightarrow \cdots \rightarrow e_{[n-k+1,n]}H_ne'_{[k+1,n]} \rightarrow 0.
	\end{multline}	
	This complex has $k+1$ non-zero terms, and the last non-zero term is in cohomological degree $n-k$. Now let $k \in \lbrace \frac{n}{2},\ldots,n\rbrace$. Then as complexes of graded $(H_{n-k,n},H_{k,n})$-bimodules we have
	\begin{multline}\label{2k>n}
		q^{\frac{(2k-n)(2k-n-1)}{2}}\Phi_n\left(\Theta 1_{-n+2k}\right) = \\
		0 \rightarrow e_{[n-k+1,k]}H_{k,n} \rightarrow \cdots \rightarrow e_{[n-k+1,r]}H_{r,n}e'_{[k+1,r]} \rightarrow \cdots \rightarrow e_{[n-k+1,n]}H_ne'_{[k+1,n]} \rightarrow 0.
	\end{multline}
	This complex has $n-k+1$ non-zero term, the last one being in cohomological degree $n-k$. In both cases, the differential is the composition of the inclusion with multiplication by the suitable idempotents. The main result regarding these complexes is the following.
	
	\begin{thm}[\cite{ChR}]\label{theta}
		Let $n \in \mathbb N$ and let $k \in \lbrace 0, \ldots, n\rbrace$. The cohomology of $\Phi_n\left(\Theta1_{-n+2k}\right)$ is concentrated in top degree $n-k$, and $H^{n-k}\left(\Phi_n\left(\Theta1_{-n+2k}\right)\right)$ is an equivalence $\cL(n)_{-n+2k} \xrightarrow{\sim} \cL(n)_{n-2k}$.
	\end{thm}

	In \cite{ChR}, Chuang and Rouquier prove this theorem for the minimal 2-representations rather than the simple 2-representations, and use it to prove that $\Theta$ provides derived equivalences on all integrable 2-representations of $\U$. In subsection \ref{actonsimp}, we give another proof of Theorem \ref{theta} in the setting of simple 2-representations. Our proof is based on finding explicit bases for the terms of $\Phi_n\left(\Theta1_{-n+2k}\right)$ and computing the action of the differential in these bases. In subsections \ref{invert} and \ref{compat}, we apply our main result on faithfulness of simple 2-representations, Theorem \ref{liftder}, to prove that $\Theta$ is invertible up to homotopy and that there are homotopy equivalences $\Theta E1_{\lambda} \simeq q^{\lambda+2}F\Theta1_{\lambda}[-1]$ for all $\lambda \in \mathbb{Z}$.

	\medskip
	
	\subsection{Action of $\Theta$ on simple 2-representations}\label{actonsimp} In this subsection, we study the complexes (\ref{2k<n}) and (\ref{2k>n}) and give the proof of Theorem \ref{theta}. To do so, we will need to compute various graded dimensions, so we start by introducing notation and recalling some elementary results.\\
	
	The graded $K$-algebras $P_n$ and $P_n^{\s_n}$ are locally finite and have graded dimensions given by
	\begin{align*}
		\mathrm{grdim}\left(P_n\right) = \frac{1}{\left(1-q^2\right)^n}, \quad \mathrm{grdim}\left(P_n^{\s_n}\right) = \frac{1}{\left(1-q^2\right)\ldots\left(1-q^{2n}\right)}.
	\end{align*}
	The following notation will be convenient: given $k \geqslant 0$, we define a variant of the quantum integer $\left \{ k\right \}$ and the quantum factorial $\left \{ k\right\}!$ by
	\[
	\left \{ k\right \} = \frac{q^{2k}-1}{q^2-1}, \quad \left \{ k\right \}! = \prod_{\ell=1}^{k} \left \{\ell\right \}.
	\]
	With these definitions, we have $\mathrm{grdim}(P_n) = \mathrm{grdim}\left(P_n^{\s_n}\right) \lbrace n \rbrace!$. More generally, given two integers $k < \ell \in \lbrace 1,\ldots, n\rbrace$, we have
	\begin{equation}\label{gdimpn}
		\mathrm{grdim}(P_n) = \mathrm{grdim}\left( P_n^{\s_{[k,\ell]}}\right) \lbrace \ell - k +1 \rbrace!.
	\end{equation}
	From $H_n \simeq \mathrm{End}_{P_n^{\s_n}}^{\bullet}(P_n)$, we deduce that the graded dimension of $H_n$ is given by
	\begin{equation}\label{grdimhn}
	\mathrm{grdim}(H_n) = \mathrm{grdim}(P_n^{\s_n})\lbrace n \rbrace!\overline{\lbrace n\rbrace!},
	\end{equation}
	where $\overline{\ \cdot \ }$ refers to the automorphism switching $q$ and $q^{-1}$ on formal series. If $M$ is a finitely generated graded $H_n$-module, then $M$ is locally finite as a $K$-vector space, and we have $M \simeq \overline{\lbrace n\rbrace!} \ e_n(M) \simeq \lbrace n \rbrace ! \ e_n'(M)$ as graded $P_n^{\s_n}$-modules. In particular, we have
	\begin{equation}\label{grdimdp}
	\mathrm{grdim}(M) = \overline{\lbrace n\rbrace!} \ \mathrm{grdim}(e_nM) = \lbrace n \rbrace! \ \mathrm{grdim}(e_n'M).
	\end{equation}
	Finally, we will need the two following formulas:
	\begin{equation}\label{hilbertsym}
		\sum_{\omega \in \s_n} q^{2l(\omega)}= \left \{ n \right \}!,
	\end{equation}
	\begin{align}\label{qbinomial}
		\sum_{0\leqslant u_1 < \ldots <u_r \leqslant n} q^{2(u_1+\ldots+u_r)} &= q^{r(r-1)}\frac{\lbrace n \rbrace!}{\lbrace r \rbrace! \lbrace n-r \rbrace!}.
	\end{align}
	The first one can be proved easily by induction on $n$, see \cite[Theorem 6.1]{qcal} for a proof of the second one.
	
	\medskip
	
	\subsubsection{Bases} Fix an integer $n \in \mathbb N$. Let $k,\ell,m$ be integers such that $k,\ell\leqslant m \leqslant n$. We construct a basis for $e_{[\ell,m]}H_{m,n}e'_{[k,m]}$ as a left $P_n^{\s[\ell,n]}$-module, which is the general form of a term of $\Phi_n(\Theta)$. Put
	\begin{align*}
		& X_{\ell,m} = \left \{ (a_{\ell},\ldots,a_m), \, 0 \leqslant a_i \leqslant n-i \right \}, \\
		& Y_{\ell,m} = \left \{ (a_{\ell},\ldots,a_m) \in X_{\ell,m}, \, a_{\ell}>\ldots>a_m\right \}. 
	\end{align*}	
	Note that these sets actually depend on $n$, that we assume fixed for the rest of this subsection. We denote by $S_{k,m}$ the set of minimal length representatives of left cosets of $\s_{[k,m]}$ in $\s_m$. For $a \in Y_{\ell,m}$ and $\omega \in S_{k,m}$ we define
	\[
		b_{m}(a,\omega) = e_{[\ell,m]}x^a\tau_{\omega}e'_{[k,m]} \in e_{[\ell,m]}H_{m,n}e'_{[k,m]},
	\]
	where $x^a = x_{\ell}^{a_l}\ldots x_{m}^{a_m}$. Our goal is to prove the following result.
	
	\begin{thm}\label{basis}
		The set $\left \{ b_m(a,\omega), \ a \in Y_{\ell,m}, \ \omega \in S_{k,m} \right \}$ is a basis of $e_{[\ell,m]}H_{m,n}e'_{[k,m]}$ as a left $P_n^{\s[\ell,n]}$-module.
	\end{thm}

	\begin{ex}
		Let us give an explicit example to illustrate the result. Assume $n=m=3$, $\ell = 1$ and $k = 2$. Then we have $Y_{1,3} = \left \{\left(2,1,0\right) \right \}$, $S_{2,3} = \left \{ 1, \ s_1, \ s_2s_1 \right \}$ and we obtain the following basis after simplification:
		\[
			\left \{ e_{[1,3]}x_1^2x_2, \ -e_{[1,3]}x_1x_2, \ e_{[1,3]}x_2 \right \}.
		\]
	\end{ex}

	We will need the following Lemma in the proof of Theorem \ref{basis}.
	
	\begin{lem}\label{step1}
		The set $\left \{ \partial_{\omega_0[\ell,m]}(x^a), \ a \in Y_{\ell,m} \right \}$ is a basis of $P_n^{\s_{[\ell,m]}\times\s_{[m+1,n]}}$ as a $P_n^{\s_{[\ell,n]}}$-module. 
	\end{lem}

	\begin{proof}
		It is well-known that the set $\left \{ x^a, \ a \in X_{\ell,m}\right \}$ is a basis of $P_n^{\s_{[m+1,n]}}$ over $P_n^{\s_{[\ell,n]}}$. Since the $P_n^{\s_{[\ell,n]}}$-linear map $\partial_{\omega_0[\ell,m]} : P_n^{\s_{[m+1,n]}} \rightarrow P_n^{\s_{[\ell,m]}\times\s_{[m+1,n]}}$ is surjective, the set $\left \{ \partial_{\omega_0[\ell,m]}(x^a), \ a \in X_{\ell,m}\right \}$ generates $P_n^{\s_{[\ell,m]}\times\s_{[m+1,n]}}$ as a $P_n^{\s_{[\ell,n]}}$-module. Furthermore, we have $\partial_{\omega_0[\ell,m]}s_i=-\partial_{\omega_0[\ell,m]}$ for $i \in \lbrace \ell,\ldots,m-1\rbrace$ and $\partial_{\omega_0[\ell,m]}(x^a)=0$ if there exists two indexes $i\neq j$ such that $a_i=a_j$. Thus if $a \in X_{\ell,m}$, the element $\partial_{\omega_0[\ell,m]}(x^a)$ is a multiple of $\partial_{\omega_0[\ell,m]}(x^b)$ for some $b \in Y_{\ell,m}$. It follows that the set $\left \{ \partial_{\omega_0[\ell,m]}(x^a), \ a \in Y_{\ell,m} \right \}$ generates $P_n^{\s_{[\ell,m]}\times\s_{[m+1,n]}}$ as a $P_n^{\s_{[\ell,n]}}$-module.
		
		To conclude, we check that the graded dimensions match. We have
		\begin{align*}
			\mathrm{grdim}\left(\bigoplus_{a \in Y_{\ell,m}} P_n^{\s_{[\ell,n]}} \partial_{\omega_0[\ell,m]}(x^a) \right) &= \mathrm{grdim}\left(P_n^{\s_{[\ell,n]}}\right)\sum_{a \in Y_{\ell,m}} q^{2(a_{\ell}+\ldots+a_m)-(m-\ell)(m-\ell+1)} \\
			&= \mathrm{grdim}\left(P_n^{\s_{[\ell,n]}}\right) \frac{\lbrace n-\ell+1\rbrace!}{\lbrace m-\ell +1\rbrace! \lbrace n-m\rbrace!} \\
			&= \mathrm{grdim}\left(P_n^{\s_{[\ell,m]}\times\s_{[m+1,n]}}\right),
		\end{align*}
		the second equality following from (\ref{qbinomial}) and the third one from (\ref{gdimpn}). Hence the result is proved.
	\end{proof}

	\begin{proof}[Proof of Theorem \ref{basis}]
		Let us start by proving that the set $\left \{ b_m(a,\omega), \ a \in Y_{\ell,m}, \ \omega \in S_{k,m} \right \}$ is free over $P_n^{\s[\ell,n]}$. Let $a \in Y_{\ell,m}$. In $H_{m,n}$ we have a decomposition of the form
		\[
		\tau_{\omega_0[\ell,m]}x^a \in \partial_{\omega_0[\ell,m]}(x^a) + \sum_{z >1} P_n^{\s_{[m+1,n]}}\tau_{z}.
		\]
		Hence for $\omega \in S_{k,m}$ we have
		\[
		e_{[\ell,m]}x^a\tau_{\omega}e'_{[k,m]} \in x^{[\ell,m]}\partial_{\omega_0[\ell,m]}(x^a)\tau_{\omega}e'_{[k,m]} + \sum_{z>\omega} P_n^{\s_{[m+1,n]}}\tau_{z}e'_{[k,m]}.
		\]
		Using Lemma \ref{step1} and the fact that the set $\left \{ \tau_ze'_{[k,m]}, \, z\in S_{k,m} \right \}$ is free over $P_n$, we obtain that the set $\left \{ b_m(a,\omega), \ a \in Y_{\ell,m}, \ \omega \in S_{k,m} \right \}$ is free over $P_n^{\s[\ell,n]}$.
		To conclude, it suffices to compute graded dimensions. On the one hand, by (\ref{grdimhn}) and (\ref{grdimdp}) we have
		\[
			\mathrm{grdim}\left(e_{[\ell,m]}H_{m,n}e'_{[k,m]}\right) = \mathrm{grdim}\left(P_n^{\s_{m}\times\s_{n-m}}\right) \frac{\lbrace m \rbrace! \overline{\lbrace m \rbrace!}}{\overline{\lbrace m-\ell+1 \rbrace!} \ \overline{\lbrace m-k+1 \rbrace!}}.
		\]
		On the other hand, the left $P_n^{\s[\ell,n]}$-submodule spanned by the set $\left \{ b_m(a,\omega), \ a \in Y_{\ell,m}, \ \omega \in S_{k,m} \right \}$ has graded dimension
		\begin{equation}\label{comp}
			\mathrm{grdim}\left(P_n^{\s_{[\ell,n]}}\right) \left(\sum_{a \in Y_{\ell,m}} q^{2(a_{\ell}+\ldots+a_m)} \right) \left( \sum_{\omega \in S_{k,m}}  q^{-2l(\omega)} \right).
		\end{equation}
		By (\ref{qbinomial}) we have
		\[
			\sum_{a \in Y_{\ell,m}} q^{2(a_{\ell}+\ldots+a_m)} = q^{(m-\ell)(m-\ell-1)}\frac{\lbrace n-\ell+1 \rbrace!}{\lbrace m-\ell+1 \rbrace! \ \lbrace n-m \rbrace!} = \frac{\lbrace n-\ell+1 \rbrace!}{\overline{\lbrace m-\ell+1 \rbrace!} \ \lbrace n-m \rbrace!}.
		\]
		By (\ref{hilbertsym}) we have
		\[
			\sum_{\omega \in S_{k,m}} q^{-2l(\omega)} = \frac{\overline{\lbrace m \rbrace!}}{\overline{\lbrace m-k+1 \rbrace!}}.
		\]
		Hence the graded dimension in equation (\ref{comp}) is equal to
		\begin{align*}
			&\mathrm{grdim}\left(P_n^{\s_{[\ell,n]}}\right)\frac{\lbrace n-\ell+1 \rbrace!}{\overline{\lbrace m-\ell+1 \rbrace!} \ \lbrace n-m \rbrace!} \ \frac{\overline{\lbrace m \rbrace!}}{\overline{\lbrace m-k+1 \rbrace!}} \\
			&= \mathrm{grdim}\left(P_n^{\s_{[1,m]}\times\s[m+1,n]}\right) \frac{\lbrace m \rbrace! \overline{\lbrace m \rbrace!}}{\overline{\lbrace m-\ell+1 \rbrace}! \ \overline{\lbrace m-k+1 \rbrace!}} \\
			&= \mathrm{grdim}\left(e_{[\ell,m]}H_{m,n}e'_{[k,m]}\right),
		\end{align*}
		and the proof is complete.
	\end{proof}

	\medskip

	\subsubsection{Combinatorics of the differential} We now study the effect of the differential of $\Phi_n\left(\Theta\right)$ on the bases given by Theorem \ref{basis}. On a general term of $\Phi_n(\Theta)$, the differential has the form
	\[
		d_m : \left \{ \begin{array}{rcl}
			e_{[\ell,m]}H_{m,n}e'_{[k,m]} & \rightarrow & e_{[\ell,m+1]}H_{m+1,n}e'_{[k,m+1]}, \\
			h & \mapsto & e_{[\ell,m+1]}he'_{[k,m+1]},
		\end{array} \right.
	\]
	where the integers $k,\ell,m$ depend on the weight and the cohomological degree. Remark that $d_m$ is a $P_n^{\s_{[\ell,n]}}$-linear map. We will explicitly determine the kernel and image of $d_m$ in terms of the bases of Theorem \ref{basis}. To do so, we prove that either $d_m(b_m(a,\omega)) =0$ or $d_m(b_m(a,\omega)) = b_{m+1}(\varphi(a,\omega))$, where $\varphi$ is an injective map that we define below. Hence the study of the map $d_m$ can be done by studying the combinatorics of the sets $Y_{\ell,m}$ and $S_{k,m}$ and the map $\varphi$. \\
	
	We now introduce the necessary notation to define the map $\varphi$. Given $r \in \lbrace 0,\ldots, m-\ell\rbrace$, we define
	\[
		Y_{m,\ell}^r = \left \{ a \in Y_{m,\ell} \, \vert \, \forall \, i \leqslant r, \, a_{m-i} =i \, \text{ and } \, a_{m-r-1} > r+1 \right \}.
	\]
	The subsets $(Y^r_{m,\ell})r$ of $Y_{m,\ell}$ are disjoint and we denote by $Y^+_{\ell,m}$ the complement of their union. More explicitly, the elements of $Y^+_{\ell,m}$ are the sequences $a \in Y_{m,\ell}$ such that $a_m>0$. Given $u \in \lbrace 1,\ldots,m\rbrace$, we denote by $S^u_{k,m}$ (resp. $S^{\geqslant u}_{k,m}$, $S^{<u}_{k,m}$) the subset of $S_{k,m}$ consisting of the elements $\omega$ such that $\omega(m)=u$ (resp. $\omega(m) \geqslant u$, $\omega(m)<u$).
	
	\begin{ex}
		Assume $n=4$, $m=3$, $k=2$ and $\ell = 1$. Then we have
		\begin{align*}
			Y_{1,3} = \left \{ (2,1,0), (3,1,0), (3,2,0), (3,2,1) \right \},
			S_{2,3} = \lbrace 1, s_1, s_2s_1 \rbrace,
		\end{align*}
		and
		\begin{align*}
			& Y_{1,3}^0 = \left \{ (3,2,0) \right \}, \quad Y^1_{1,3} = \left \{ (3,1,0) \right \}, \\
			& Y_{1,3}^2 = \left \{ (2,1,0) \right \}, \quad Y^{+}_{1,3} = \left \{ (3,2,1) \right \}, \\
			& S_{2,3}^3 = \lbrace 1, s_1 \rbrace, \quad S_{2,3}^2 = \lbrace s_2s_1 \rbrace.
		\end{align*}
	\end{ex}
	
	We now define a map $\varphi_Y : Y_{\ell,m} \rightarrow Y_{\ell,m+1}$. If $a\in Y^+_{\ell,m}$ we put $\varphi_Y(a)=(a,0)$. If $a\in Y^r_{\ell,m}$ we put $\varphi_Y(a)=(a_{\ell},\ldots,a_{m-r-1},r+1,r,\ldots,0)$ (so we have increased the entries $a_{m-r},\ldots,a_m$ by 1 and added a 0 as the last entry). 
	We can now define the map $\varphi$ as follows
	\[
		\varphi : \left \{ \begin{array}{rcl}
			\left(Y_{\ell,m}^{+}\times S_{k,m}\right) \bigcup \left(\displaystyle\bigcup_{r=0}^{m-\ell} \left(Y_{\ell,m}^r\times S_{k,m}^{<m-r}\right) \right) & \rightarrow & Y_{\ell,m+1}\times S_{k,m+1} \\
			(a,\omega) & \mapsto & \left \{ \begin{array}{ll}
				(\varphi_Y(a),\omega) & \text{if} \ a \in Y_{\ell,m}^+, \\
				(\varphi_Y(a),s_{m-r}\ldots s_m\omega) & \text{if} \ a \in Y_{\ell,m}^r.
			\end{array} \right.
		\end{array} \right.
	\] 
	
	\begin{prop}\label{phi}
		The map $\varphi$ is injective and has image
		\[
			\mathrm{im}(\varphi) = \bigcup_{r=0}^{m+1-\ell} Y^r_{\ell,m+1}\times S_{k,m+1}^{\geqslant m+1-r}.
		\]
	\end{prop}

	\begin{proof}
		The map $\varphi$ sends $Y^+_{\ell,m}\times S_{k,m}$ to $Y_{\ell,m+1}\times S_{k,m+1}^{m+1}$ and $Y^r_{\ell,m}\times S_{k,m}^{<m-r}$ to $Y_{\ell,m+1}\times S_{k,m+1}^{m-r}$. Since the $S_{k,m+1}^u$ are disjoint for distinct values of $u$, it suffices to prove that $\varphi_{\vert Y_{\ell,m}^+\times S_{k,m}}$ and $\varphi_{\vert Y^r_{\ell,m}\times S_{k,m}^{<m-r}}$ are injective. However it is clear that the maps $(S_{k,m}\rightarrow S_{k,m+1}, \omega \mapsto \omega)$ and $(S_{k,m}^{<m-r} \rightarrow S_{k,m+1}, \omega \mapsto s_{m-r}\ldots s_m\omega)$ are injective. Hence $\varphi$ is injective.
		
		We now compute $\mathrm{im}(\varphi)$. We have
		\[
			\varphi\left(Y_{\ell,m}^+\times S_{k,m}\right) = \bigcup_{r'=0}^{m+1-\ell} Y_{\ell,m+1}^{r'}\times S_{k,m+1}^{m+1},
		\]
		and
		\[
			\varphi\left(Y_{\ell,m}^{r}\times S_{k,m}^{<m-r}\right) = \bigcup_{r'=r+1}^{m+1-\ell} Y_{\ell,m+1}^{r'}\times S_{k,m+1}^{m-r}.
		\]
		Hence
		\[
			\mathrm{im}(\varphi) = \left(\bigcup_{r'=0}^{m+1-\ell} Y_{\ell,m+1}^{r'}\times S_{k,m+1}^{m+1} \right)\bigcup \left( \bigcup_{r=0}^{m-\ell}\left( \bigcup_{r'=r+1}^{m+1-\ell} Y_{\ell,m+1}^{r'}\times S_{k,m+1}^{m-r}\right) \right).
		\]
		Switching the order of the two unions in the second term gives
		\[
			\mathrm{im}(\varphi) = \left(\bigcup_{r'=0}^{m+1-\ell} Y_{\ell,m+1}^{r'}\times S_{k,m+1}^{m+1} \right)\bigcup \left( \bigcup_{r'=1}^{m+1-\ell} \left( \bigcup_{r=0}^{r'-1}Y_{\ell,m+1}^{r'}\times S_{k,m+1}^{m-r}\right) \right).
		\]
		We can isolate the term $r'=0$ in the first union and merge the two remaining unions to obtain
		\begin{align*}
			\mathrm{im}(\varphi) &= \left(Y_{\ell,m+1}^{0}\times S_{k,m+1}^{m+1} \right) \bigcup \left(\bigcup_{r'=1}^{m+1-\ell} \left( \bigcup_{r=0}^{r'}Y_{\ell,m+1}^{r'}\times S_{k,m+1}^{m+1-r}\right) \right) \\
			&=\bigcup_{r'=0}^{m+1-\ell} Y^{r'}_{\ell,m+1}\times S_{k,m+1}^{\geqslant m+1-r'}.
		\end{align*}
	\end{proof}
	
	\begin{prop}\label{diff}
		Let $(a,\omega) \in Y_{\ell,m}\times S_{k,m}$. Then we have
		\[
			d_m(b_m(a,\omega)) = \left \{ \begin{array}{ll}
				0 & \text{if} \ (a,\omega) \in \displaystyle \bigcup_{r=0}^{m-\ell} Y^r_{\ell,m}\times S_{k,m}^{\geqslant m-r}, \\
				b_{m+1}(\varphi(a,\omega)) & \text{otherwise}.
			\end{array} \right.
		\]
	\end{prop}

	\begin{proof}
		We have $d_m(b_m(a,\omega)) = e_{[\ell,m+1]}x^a\tau_{\omega}e'_{[k,m+1]}$. When $a \in Y^+_{\ell,m}$, it is clear that $d_m(b_m(a,\omega)) = b_{m+1}(\varphi(a,\omega))$. Assume now that $a\in Y^r_{\ell,m}$. In $P_n$, we have
		\[
			\partial_m\ldots\partial_{m-r}\left(x_{m-r}^{r+1}x_{m-r+1}^{r}\ldots x_m\right) = (-1)^{r+1}x_{m-r}^rx_{m-r+1}^{r-1}\ldots x_{m-1}.
		\]
		Furthermore, if $w \in \mathfrak S_{[\ell,m+1]}$ and $P \in P_n$ we have $\tau_{\omega_0[\ell,m+1]}P\tau_w = (-1)^{l(w)}\tau_{\omega_0[\ell,m+1]}\partial_{w^{-1}}(P)$. It follows that
		\begin{align*}
			e_{[\ell,m+1]} x^a &= e_{[\ell,m+1]} x^ax_{m-r}x_{m-r-1}\ldots x_m \tau_{m-r}\tau_{m-r+1}\ldots\tau_m\\
			&= e_{[\ell,m+1]}x^{\varphi_Y(a)} \tau_{m-r}\ldots\tau_m.
		\end{align*}
		Hence
		\begin{align*}
			d_m(b_m(a,\omega)) &= e_{[\ell,m+1]}x^{\varphi_Y(a)} \tau_{m-r}\ldots\tau_m\tau_{\omega}e'_{[k,m+1]}.
		\end{align*}
		If $\omega(m) \geqslant m-r$, then $s_{m-r}\ldots s_m\omega$ does not have minimal length in its left coset modulo $\s_{[k+1,m]}$. In that case, $\tau_{m-r}\ldots\tau_m\tau_{\omega}e'_{[k,m+1]}=0$, and we conclude that $d_m(b_m(a,\omega))=0$. Otherwise if $\omega(m) <m-r$, then $s_{m-r}\ldots s_m\omega$ has minimal length in its left coset modulo $\s_{[k+1,m+1]}$ and $\tau_{m-r}\ldots \tau_{m}\tau_{\omega}=\tau_{s_{m-r}\ldots s_m\omega}$. If follows that $d_m(b_m(a,\omega)) = b_{m+1}(\varphi(a,\omega))$.
	\end{proof}

	From Propositions \ref{phi} and \ref{diff}, we obtain bases for $\mathrm{ker}(d_m)$ and $\mathrm{im}(d_m)$ and deduce that the cohomology of $\Phi_n(\Theta)$ is concentrated in top degree, which is the first part of Theorem \ref{theta}.

	\begin{cor}\label{kerim}
		The sets
		\begin{align*}
			&\left \{ d_m(a,\omega), \ (a,\omega) \in \bigcup_{r=0}^{m-\ell} Y^r_{\ell,m}\times S_{k,m}^{\geqslant m-r} \right \}, \\
			&\left \{ d_{m+1}(a,\omega), \ (a,\omega) \in \bigcup_{r=0}^{m+1-\ell} Y^r_{\ell,m+1}\times S_{k,m+1}^{\geqslant m+1-r} \right \},
		\end{align*}
		are bases of $\mathrm{ker}(d_m)$, $\mathrm{im}(d_m)$ respectively, as left $P_n^{\s_{[\ell,n]}}$-modules.
	\end{cor}

	\begin{cor}\label{top}
		For all $n \in \mathbb{N}$, $k \in \lbrace 0,\ldots,n \rbrace$ and $\ell \neq n-k$ we have:
		\[
			H^{\ell}\left(\Phi_n\left(\Theta1_{-n+2k}\right) \right) = 0.
		\]
	\end{cor}

	\begin{proof}
		By Corollary \ref{kerim} we have $\mathrm{im}(d_m) = \mathrm{ker}(d_{m+1})$ if $m+1<n$. The result follows.
	\end{proof}
	
	\medskip

	\subsection{Invertibility of $\Theta$}\label{invert}
	\subsubsection{Top cohomology} Let $n\in \mathbb{N}$ and $k \in \lbrace 0,\ldots,n\rbrace$. We give a description of $H^{n-k}\left(\Phi_n\left(\Theta1_{-n+2k}\right)\right)$ as $(H_{n-k,n},H_{k,n})$-bimodule. Using Corollary \ref{kerim}, we can give a basis for $H^{n-k}\left(\Phi_n\left(\Theta1_{-n+2k}\right)\right)$ as a left $P_n^{\s_{[n-k+1,n]}}$-module. With the notations introduced previously, we have $Y_{n-k+1,n} = \lbrace (k-1,\ldots,0) \rbrace$ and $S_{k,n}^{<n-k+1} = \lbrace \sigma_k\omega, \ \omega \in \s_k \rbrace$, where $\sigma_k$ the element of $\s_n$ sending $i$ to $i+n-k$ if $i \leqslant k$, and to $i-k$ otherwise. The picture below depicts the element $\sigma_k$ as a strand diagram.
	
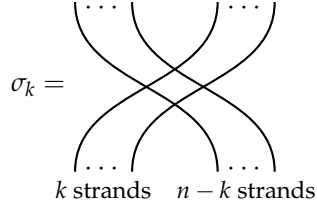
\[
		\begin{tikzpicture}[scale=0.75]
		\draw [thick] (0,0) to [out=90,in=-90]  (2.5,3);
		\draw [thick] (1,0) to [out=90,in=-90]  (3.5,3);
		\draw [thick] (2.5,0) to [out=90,in=-90] (0,3);
		\draw [thick] (3.5,0) to [out=90,in=-90] (1,3);
		\node at (0.5,0.05) {$\cdots$};
		\node at (3,0.05) {$\cdots$};
		\node at (3,2.9) {$\cdots$};
		\node at (0.5,2.9) {$\cdots$};	
		\node [below] at (0.5,0) {\small{$k$ strands}};
		\node [below] at (3, 0) {\small{$n-k$ strands}};
		\node [left] at (0, 1.5) {$\sigma_k =$};
		\end{tikzpicture}
	\]
	Then there is a decomposition as left $P_n^{\s_{[n-k+1,n]}}$-module:
	\begin{equation}\label{cohobasis}
		q^{\frac{(2k-n)(2k-n-1)}{2}}H^{n-k}\left(\Phi_n\left(\Theta1_{-n+2k}\right)\right) = \bigoplus_{\omega \in \s_k} P_n^{\s_{[n-k+1,n]}} e_{[n+1-k,n]} x_{n+1-k}^{k-1}\ldots x_{n-1} \tau_{\sigma_k}\tau_{\omega} e'_{[k+1,n]}.
	\end{equation}	
	 To describe the bimodule structure, it is more easier to work in terms of $\left(P_n^{\s_{n-k}\times\s_{k}},P_n^{\s_{k}\times\s_{n-k}}\right)$-bimodules using the Morita equivalences from (\ref{morita}).
	
	\begin{thm}\label{coho}
		There is an integer $d$ and an isomorphism of graded $\left(P_n^{\s_{n-k}\times\s_{k}},P_n^{\s_{k}\times\s_{n-k}}\right)$-bimodules
		\[
			e_{[1,n-k]}H^{n-k}\left(\Phi_n\left(\Theta1_{-n+2k}\right)\right)e'_{[1,k]} \simeq q^{d}P_n^{\s_{n-k}\times\s_k},
		\]
		where the left action of $P_n^{\s_{n-k}\times s_k}$ is given by multiplication and the right action of $P_n^{\s_k\times\s_{n-k}}$ is given by $\sigma_k$.
	\end{thm}

	\begin{proof}
		Define
		\[
			b'=x_{[1,n-k]}x_{[n-k+1,n]}\tau_{\omega_0[1,n]}x'_{[1,k]}x'_{[k+1,n]} \in e_{[1,n-k]}e_{[n-k+1,n]}H_ne'_{[1,k]}e'_{[k+1,n]}.
		\]
		This is an element of $e_{[1,n-k]}\Phi_n\left(\Theta^{n-k}1_{-n+2k}\right)e'_{[1,k]}$. Let $b$ be the class of $b'$ in $e_{[1,n-k]}H^{n-k}\left(\Phi_n\left(\Theta1_{-n+2k}\right)\right)e'_{[1,k]}$. Consider the $P_n^{\s_{n-k}\times\s_k}$-linear map defined by
		\[
			f : \left \{ \begin{array}{rcl}
				P_n^{\s_{n-k}\times\s_k} & \rightarrow & e_{[1,n-k]}H^{n-k}\left(\Phi_n\left(\Theta1_{-n+2k}\right)\right)e'_{[1,k]}, \\
				1 & \mapsto & b.
			\end{array} \right.
		\]
		Let us show that $f$ is surjective. If $\omega \in \s_k\setminus\lbrace 1\rbrace$, we have $\tau_{\omega}e'_{[1,k]}=0$. Hence decomposition (\ref{cohobasis}) yields
		\[
			e_{[1,n-k]}H^{n-k}\left(\Phi_n\left(\Theta1_{-n+2k}\right)\right)e'_{[1,k]} = e_{[1,n-k]}P_n^{\s_{[n-k+1,n]}} e_{[n-k+1,n]} x_{n-k+1}^{k-1}\ldots x_{n-1} \tau_{\sigma_k}e'_{[1,k]}e'_{[k+1,n]}.
		\]
		Note that $e_{[n-k+1,n]} x_{n-k+1}^{k-1}\ldots x_{n-1} \tau_{\sigma_k}e'_{[1,k]}e'_{[k+1,n]} = \pm x_{[n-k+1,n]}\tau_{\omega_0[1,n]}x'_{[1,k]}x'_{[k+1,n]}$. If follows that for $P \in P_n^{\s_{[n-k+1,n]}}$ we have
		\[
			e_{[1,n-k]}Pe_{[n+1-k,n]} x_{n+1-k}^{k-1}\ldots x_{n-1} \tau_{\sigma_k}e'_{[1,k]}e'_{[k+1,n]} = \pm \partial_{\omega_0[1,n-k]}(P)b'.
		\]
		Thus $f$ is surjective. To conclude that $f$ is an isomorphism, we compute graded dimensions. By decomposition (\ref{cohobasis}), we have
		\[
			\mathrm{grdim}\left(H^{n-k}\left(\Phi_n\left(\Theta1_{-n+2k}\right)\right)\right) = \mathrm{grdim}\left(P_n^{\s_{[n-k+1,n]}}\right) q^{k(k-1)-2k(n-k)}\overline{\lbrace k \rbrace!}.
		\]
		Hence we deduce
		\begin{align*}
			\mathrm{grdim}\left(e_{[1,n-k]}H^{n-k}\left(\Phi_n\left(\Theta1_{-n+2k}\right)\right)e'_{[1,k]}\right) &= \frac{\mathrm{grdim}\left(P_n^{\s_{[n-k+1,n]}}\right) q^{k(k-1)-2k(n-k)}}{\overline{\lbrace n-k \rbrace!}} \\
			&= \mathrm{grdim}\left( P_n^{\s_{n-k}\times\s_k}\right) q^{(n-k)(n-k-1)+k(k-1)-2k(n-k)} \\
			&= \mathrm{grdim}\left( P_n^{\s_{n-k}\times\s_k}\right) q^{\mathrm{deg}(b)}.
		\end{align*}
		So $f$ is an isomorphism of left $P_n^{\s_{n-k}\times\s_k}$-modules. All that remains to prove is the compatibility of $f$ with the right $P_n^{\s_{k}\times\s_{n-k}}$-module structure, which we do in Lemma \ref{right} below.
	\end{proof}

	\begin{lem}\label{right}
		Keep the notations of the proof of Theorem \ref{coho}. For all $P \in P_n^{\s_{k}\times\s_{n-k}}$, we have $bP=\sigma_k(P)b$.
	\end{lem}

	\begin{proof}
		We prove that $b'P - \sigma_k(P)b'$ is a coboundary of $\Phi_n\left(\Theta1_{-n+2k}\right)$. Using the relations of the affine nil Hecke algebra, we have
		\[
			b'P-\sigma_k(P)b' \in \sum_{\substack{\omega \in \s_n/\s_k\times\s_{n-k} \\ \omega <\sigma_k}} P_n\tau_{\omega}e'_{[1,k]}e'_{[k+1,n]}.
		\]
		Since $b'=e_{[n-k+1,n]}b'$, we actually have
		\begin{equation}\label{rightact}
			b'P-\sigma_k(P)b' \in \sum_{\substack{\omega \in \s_n/\s_k\times\s_{n-k} \\ \omega <\sigma_k}} e_{[n-k+1,n]}P_n\tau_{\omega}e'_{[1,k]}e'_{[k+1,n]}.
		\end{equation}
		Let us explain why the right-hand side of (\ref{rightact}) is contained in the image of the differential of $\Phi_n\left(\Theta1_{-n+2k}\right)$. Since $\left \{ \partial_{\omega}\left(x_{n-k+1}^{k-1}\ldots x_{n-1}\right), \ \omega \in \s_{[n-k+1,n]} \right \}$ is a basis for $P_n$ as a $P_n^{\s_{[n-k+1,n]}}$-module, we have:
		\[
			e_{[n-k+1,n]}P_n = \bigoplus_{\omega \in \s_{[n-k+1,n]}} P_n^{\s_{[n-k+1,n]}}e_{[n-k+1,n]}x_{n-k+1}^{k-1}\ldots x_{n-1} \tau_{\omega}.
		\]
		Hence the right-hand side of (\ref{rightact}) has the form
		\[
			\sum_{\substack{\omega \in \s_n/\s_k\times\s_{n-k}, \ \omega'\in \s_{[n-k+1,n]} \\ \omega <\sigma_k}} P_n^{\s_{[n-k+1,n]}}e_{[n-k+1,n]}x_{n-k+1}^{k-1}\ldots x_{n-1} \tau_{\omega'}\tau_{\omega}e'_{[1,k]}e'_{[k+1,n]}.
		\]
		For $\omega \in \s_n/\s_k\times\s_{n-k}$, the condition $\omega <\sigma_k$ is equivalent to $\omega(n) > n-k$. If furthermore $\omega' \in \s_{[n-k+1,n]}$, we have $\omega'\omega(n) >n-k$. Hence on the basis of Theorem \ref{basis}, the right-hand side of (\ref{rightact}) decomposes as
		\[
			\bigoplus_{\substack{\omega \in \s_n/\s_k\times\s_{n-k} \\ \omega <\sigma_k}} P_n^{\s_{[n-k+1,n]}}e_{[n-k+1,n]}x_{n-k+1}^{k-1}\ldots x_{n-1}\tau_{\omega\omega_0[1,k]}e'_{[k+1,n]}.
		\]
		By Corollary \ref{kerim}, this is contained in the image of the differential of $\Phi_n\left(\Theta1_{-n+2k}\right)$. Hence $	b'P-\sigma_k(P)b'$ is a coboundary, and the result follows.
	\end{proof}

	Theorem \ref{coho} implies in particular that $H^{n-k}\left(\Phi_n\left(\Theta1_{-n+2k}\right)\right)$ is an invertible bimodule, which proves the second part of Theorem \ref{theta}.

	\subsubsection{Invertibility of $\Theta$} We denote by $\Theta^{\vee}$ the right dual of $\Theta$.
	
	\begin{thm}\label{thetainv}
		For all $\lambda \in \mathbb{Z}$, the unit $u_{\lambda} : 1_{\lambda} \rightarrow \Theta^{\vee}\Theta1_{\lambda}$ and counit $v_{\lambda} : \Theta\Theta^{\vee}1_{-\lambda} \rightarrow 1_{-\lambda}$ of adjunction are homotopy equivalences. In other words, $\Theta$ is invertible in $K(\U)$.
	\end{thm}

	\begin{proof}
		If $M$ is an object of an integrable 2-representation of $\U$, we have $E^i(M) = F^i(M)=0$ for $i$ large enough. Hence the complexes $\Theta\Theta^{\vee}(M)$ and $\Theta^{\vee}\Theta(M)$ are bounded. So by Theorem \ref{liftder}, it suffices to check that the for all $n \in \mathbb{N}$, $\Phi_n(u_{\lambda})$ and $\Phi_n(v_{\lambda})$ are quasi-isomorphisms.
		
		Let $n \in \mathbb{N}$. If $\lambda \notin \lbrace -n +2k, \, k \in \lbrace0,\ldots,n\rbrace\rbrace$, then $\Phi_n(1_{\lambda}) = \Phi_n(1_{-\lambda}) = 0$, so the result is clear. If $k \in  \lbrace0,\ldots,n\rbrace$, by Corollary \ref{top} there is a quasi-isomorphism
		\[
			\Phi_n\left(\Theta1_{-n+2k}\right) \simeq H^{n-k}\left(\Phi_n\left(\Theta1_{-n+2k}\right)\right)[n-k].
		\]
		By Theorem \ref{coho},  $H^{n-k}\left(\Phi_n\left(\Theta1_{-n+2k}\right)\right)$ is invertible as a $(H_{n-k,n},H_{k,n})$-bimodule. It follows that the complex $\Phi_n\left(\Theta1_{-n+2k}\right)$ is invertible in $D(\cL(n)\mathrm{-bim})$. Thus $\Phi_n(u_{-n+2k})$ and $\Phi_n(v_{-n+2k})$ are quasi-isomorphisms, which ends the proof.
	\end{proof}
	
	\medskip

	\subsection{Compatibility with Chevalley generators}\label{compat}
	
	The goal of this subsection is to prove that for all $\lambda \in \mathbb Z$, there is a homotopy equivalence $\Theta E1_{\lambda} \simeq q^{\lambda + 2}F\Theta1_{\lambda}[-1]$. To prove this, we start by constructing a map $G_{\lambda} : \Theta E1_{\lambda} \rightarrow q^{\lambda + 2}F\Theta1_{\lambda}[-1]$. We will then prove that this is a homotopy equivalence using Theorem \ref{liftder}.\\
	
	 Given $k,\ell \in \mathbb{N}$, we define an element $G_{k,\ell} \in H_k \otimes H_{\ell}^{\mathrm{op}}$ of degree $2(\ell-k)$ by
	\[
		G_{k,\ell} = \sum_{r=0}^{\ell-1} (-1)^rx_1^rx_{[2,k]}\tau_{\omega_0[1,k]} \otimes \epsilon_{\ell-1-r}(x_2,\ldots,x_{\ell})e'_{\ell},
	\]
	where $\epsilon_j$ denotes the elementary symmetric polynomial of degree $j$. In $H_k\otimes H_{\ell}^{\mathrm{op}}$, we have $G_{k,\ell} = G_{k,\ell}(e_{k}\otimes e'_{[2,\ell]}) = (e_{[2,k]}\otimes e'_{\ell})G_{k,\ell}$. Hence $G_{k,\ell}$ defines a map $G_{k,\ell} : F^{(k)}E^{(\ell-1)}E \rightarrow q^{k-\ell}FF^{(k-1)}E^{(\ell)}$.
	
	\begin{lem}\label{morphcomp}
		For all $k,\ell \in \mathbb N$, the following diagram commutes:
		\[
		\begin{tikzcd}
			F^{(k)}E^{(\ell-1)}E \arrow[r, "{d_{k,\ell-1}E}"] \arrow[d, "{G_{k,\ell}}"'] & F^{(k+1)}E^{(\ell)}E \arrow[d, "{G_{k+1,\ell+1}}"] \\
			FF^{(k-1)}E^{(\ell)} \arrow[r, "{Fd_{k-1,\ell}}"']                           & FF^{(k)}E^{(\ell+1)}                              
		\end{tikzcd}
		\]
		where $d_{k,\ell} : F^{(k)}E^{(\ell)} \rightarrow F^{(k+1)}E^{(\ell+1)}$ is the composition of the unit of adjunction with the projection on the idempotent $e_{k+1} \otimes e'_{\ell}$.
	\end{lem}

	\begin{proof}
		This is a straightforward computation. On the one hand, we have
		\begin{align*}
			Fd_{k-1,\ell} \circ G_{k,\ell} &=\left(e_{[2,k+1]}\otimes e'_{\ell+1}\right) \circ F^k\eta E^{\ell} \circ \sum_{r=0}^{\ell-1} (-1)^rx_1^rx_{[2,k]}\tau_{\omega_0[1,k]} \otimes \epsilon_{\ell-1-r}(x_2,\ldots,x_{\ell})e'_{\ell} \\
			&= \left( \sum_{r=0}^{\ell-1} (-1)^re_{[2,k+1]}x_1^rx_{[2,k]}\tau_{\omega_0[1,k]}\otimes \epsilon_{\ell-1-r}(x_2,\ldots,x_{\ell})e'_{\ell+1}\right) \circ F^k\eta E^{\ell}.
		\end{align*}
		In $H_{k+1}$ we have $e_{[2,k+1]}x_1^rx_{[2,k]}\tau_{\omega_0[1,k]} = x_1^{r}x_{[2,k+1]}\tau_{s_1\omega_0[1,k+1]}$. Hence
		\[
			Fd_{k-1,\ell} \circ G_{k,\ell} = \left( \sum_{r=0}^{\ell-1} (-1)^r x_1^{r}x_{[2,k+1]}\tau_{s_1\omega_0[1,k+1]}\otimes \epsilon_{\ell-1-r}(x_2,\ldots,x_{\ell})e'_{\ell+1}\right) \circ F^k\eta E^{\ell}.
		\]
		On the other hand, we have
		\[
			G_{k+1,\ell+1}\circ d_{k,\ell-1}E = \left(\sum_{r=0}^{\ell}(-1)^rx_1^rx_{[2,k+1]}\tau_{\omega_0[1,k+1]} \otimes \epsilon_{\ell-r}(x_2,\ldots x_{\ell+1})e'_{\ell+1}\right) \circ F^k\eta E^{\ell}.
		\]
		The elementary symmetric polynomials satisfy the relation
		\[
			\epsilon_{\ell-r}(x_2,\ldots x_{\ell+1}) = \epsilon_{\ell-r}(x_2,\ldots x_{\ell}) + x_{\ell+1}\epsilon_{\ell-r-1}(x_2\ldots,x_{\ell}),
		\]
		and we can rewrite the above as
		\begin{align*}
			G_{k+1,\ell+1}\circ d_{k,\ell-1}E = & \bigg(\sum_{r=0}^{\ell-1}(-1)^rx_1^rx_{[2,k+1]}\tau_{\omega_0[1,k+1]} \otimes x_{\ell+1}\epsilon_{\ell-r-1}(x_2,\ldots x_{\ell})e'_{\ell+1}  \\
			& + \sum_{r=1}^{\ell}(-1)^rx_1^rx_{[2,k+1]}\tau_{\omega_0[1,k+1]} \otimes \epsilon_{\ell-r}(x_2,\ldots x_{\ell})e'_{\ell+1}\bigg) \circ F^k\eta E^{\ell}.
		\end{align*}
		We have $Fx\circ \eta = xE\circ \eta$. Using this, we can take the $x_{\ell+1}$ in the first sum to the right left side of the tensor and we obtain
		\begin{align*}
			G_{k+1,\ell+1}\circ d_{k,\ell-1}E =& \bigg(\sum_{r=0}^{\ell-1}(-1)^rx_1^rx_{[2,k+1]}\tau_{\omega_0[1,k+1]}x_{k+1}\otimes \epsilon_{\ell-r-1}(x_2,\ldots x_{\ell})e'_{\ell+1}  \\
			& - \sum_{r=0}^{\ell-1}(-1)^rx_1^{r+1}x_{[2,k+1]}\tau_{\omega_0[1,k+1]} \otimes \epsilon_{\ell-r-1}(x_2,\ldots x_{\ell})e'_{\ell+1}\bigg) \circ F^k\eta E^{\ell} \\
			=& \bigg( \sum_{r=0}^{\ell-1} (-1)^rx_1^rx_{[2,k+1]}\left(\tau_{\omega_0[1,k+1]}x_{k+1}-x_1\tau_{\omega_0[1,k+1]}\right)\otimes \epsilon_{\ell-r-1}(x_2,\ldots x_{\ell})e'_{\ell+1} \bigg) \circ F^k\eta E^{\ell}.
		\end{align*}
		In $H_{k+1}$, we have the relation $\tau_{\omega_0[1,k+1]}x_{k+1}-x_1\tau_{\omega_0[1,k+1]}=\tau_{s_1\omega_0[1,k+1]}$, and this completes the proof.
	\end{proof}

	We can now define the morphism $G_{\lambda} : \Theta E 1_{\lambda}\rightarrow q^{\lambda +2}F\Theta1_{\lambda}[-1]$. On the $r^{\mathrm{th}}$ component of $\Theta E 1_{\lambda}$, $G_{\lambda}$ is defined to be the map $G_{\lambda+r+2,r+1} : q^{-r}F^{(\lambda+r+2)}E^{(r)}E1_{\lambda} \rightarrow q^{\lambda+1-r}FF^{(\lambda+r+1)}E^{(r+1)}1_{\lambda}$. This gives a morphism of complexes by Lemma \ref{morphcomp}. Similarly, we define a morphism of complexes $T_{\lambda} : q^{\lambda + 2}F\Theta1_{\lambda}[-1] \rightarrow \Theta E1_{\lambda}$ using the elements
	\[
		T_{k,\ell} = \sum_{r=0}^{k-1} (-1)^re_k \epsilon_{k-1-r}(x_2,\ldots,x_k) \otimes \tau_{\omega_0[1,\ell]}x_1^rx'_{[2,\ell]} \in H_k\otimes H_{\ell}^{\mathrm{op}}.
	\]
	The fact that these give a morphism of complexes is proved as for the map $G_{\lambda}$.
	
	\begin{lem}
		Let $k,\ell \in \mathbb N$. If $\ell \leqslant k$, then $T_{k,\ell}G_{k,\ell}=e_k\otimes e'_{[2,\ell]}$. If $k \leqslant \ell$, then $G_{k,\ell}T_{k,\ell}=e_{[2,k]}\otimes e'_{\ell}$.
	\end{lem}

	\begin{proof}
		We prove the result in the case $\ell \leqslant k$, the other one being similar. We have
		\begin{align*}
			T_{k,\ell}G_{k,\ell} &= \sum_{\substack{0\leqslant r \leqslant k-1 \\ 0\leqslant u \leqslant \ell -1}} (-1)^{r+u} e_k \epsilon_{k-1-r}(x_2,\ldots,x_k) x_1^ux_{[2,k]}\tau_{\omega_0[1,k]} \otimes \epsilon_{\ell-1-u}(x_2,\ldots,x_{\ell})e'_{\ell} \tau_{\omega_0[1,\ell]}x_1^rx'_{[2,\ell]}.
		\end{align*}
		In $H_k$ we have
		\[
			e_k \epsilon_{k-1-r}(x_2,\ldots,x_k) x_1^ux_{[2,k]}\tau_{\omega_0[1,k]} = \partial_{k-1}\ldots\partial_1\left(\epsilon_{k-1-r}(x_2,\ldots,x_k) x_1^u \right) e_k = (-1)^{u}\delta_{u,r}e_k.
		\]
		Hence
		\begin{align*}
			T_{k,\ell}G_{k,\ell} &= \sum_{u=0}^{\ell-1} (-1)^{u} e_k\otimes \epsilon_{\ell-1-u}(x_2,\ldots,x_{\ell})e'_{\ell} \tau_{\omega_0[1,\ell]}x_1^ux'_{[2,\ell]} \\
			&=\sum_{u=0}^{\ell-1} (-1)^{u} e_k\otimes \epsilon_{\ell-1-u}(x_2,\ldots,x_{\ell})\tau_{\ell-1}\ldots\tau_1x_1^ue'_{[2,\ell]}.
		\end{align*}
		Using the polynomial representation, it is easy to check that
		\[
			\sum_{u=0}^{\ell-1} (-1)^u \epsilon_{\ell-1-u}(x_2,\ldots,x_{\ell})\tau_{\ell-1}\ldots\tau_1x_1^ue'_{[2,\ell]} = e'_{[2,\ell]},
		\]
		and the result follows.
	\end{proof}

	\begin{cor}\label{surjinj}
		If $\lambda \geqslant 0$, then $T_{\lambda}G_{\lambda} = \mathrm{id}_{\Theta E1_{\lambda}}$. If $\lambda \leqslant 0$, then $G_{\lambda}T_{\lambda} = \mathrm{id}_{F\Theta1_{\lambda}[-1]}$.
	\end{cor}

	\begin{thm}\label{thetae}
		For all $\lambda \in \mathbb{Z}$, $G_{\lambda} : \Theta E1_{\lambda} \rightarrow q^{\lambda+2} F\Theta1_{\lambda}[-1]$ is a homotopy equivalence.
	\end{thm}

	\begin{proof}
		If $M$ is an object of an integrable 2-representation of $\U$, the complexes $\Theta E(M)$ and $F\Theta(M)$ are bounded. Thus by Theorem \ref{liftder} it suffices to prove that for all $n \in \mathbb{N}$, $\Phi_n(G_{\lambda})$ is a quasi-isomorphism.
		
		Let $n\in \mathbb{N}$ and $k \in \lbrace 0,\ldots,n\rbrace$. There is an isomorphism of $\left(P_n^{\s_{k+1}\times\s_{n-k-1}},P_n^{\s_k\times\s_{n-k}}\right)$-bimodules
		\[
			e_{[1,k+1]}\Phi_n\left( E1_{-n+2k}\right)e'_{[1,k]} \simeq P_n^{\s_{[1,k]}\times\s_{[k+2,n]}}.
		\]
		Hence by Corollary \ref{top} and Proposition \ref{coho} there is an isomorphism of $\left(P_n^{\s_{n-k-1}\times\s_{k+1}},P_n^{\s_k\times\s_{n-k}}\right)$-bimodules
		\[
			e_{[1,n-k-1]} H^{\ell} \left( \Phi_n \left( \Theta E1_{-n+2k} \right) \right) e'_{[1,k]} = \left \{ \begin{array}{cl}
			P_n^{\s_{[1,k]}\times\s_{[k+2,n]}} & \text{ if } \ell = n-k-1, \\
			0 & \text{ otherwise},
			\end{array} \right.
		\]
		where the left action is given by $\sigma_{k+1}^{-1}$ in the first case. Similarly, there is an isomorphism
		\[
			e_{[1,n-k-1]} H^{\ell} \left( \Phi_n \left( F\Theta1_{-n+2k} \right) \right) e'_{[1,k]} = \left \{ \begin{array}{cl}
			P_n^{\s_{[1,n-k-1]}\times\s_{[n-k+1,n]}} & \text{ if } \ell = n-k, \\
			0 & \text{ otherwise},
			\end{array} \right.
		\]
		where the right action is given by $\sigma_k$ in the first case. Now remark that there is an isomorphism of $\left(P_n^{\s_{n-k-1}\times\s_{k+1}},P_n^{\s_k\times\s_{n-k}}\right)$-bimodules
		\[
			s_{n-k-1}\ldots s_1\sigma_k : P_n^{\s_{[1,k]}\times\s_{[k+2,n]}} \xrightarrow{\sim} P_n^{\s_{[1,n-k-1]}\times\s_{[n-k+1,n]}}.
		\]
		Hence the cohomology bimodules of $\Phi_n\left(\Theta E1_{\lambda}\right)$ and $\Phi_n\left(F\Theta 1_{\lambda}[-1]\right)$ are isomorphic. Furthermore by Lemma \ref{surjinj}, for all $\ell$, $H^{\ell}\left(\Phi_n(G_{\lambda})\right)$ is either a split surjection or a split injection. Since the source and target of $H^{\ell}\left(\Phi_n(G_{\lambda})\right)$ are locally finite and have the same graded dimension, we conclude that $H^{\ell}(\Phi_n(G_{\lambda}))$ is an isomorphism. Hence $\Phi_n(G_{\lambda})$ is a quasi-isomorphism, and the proof is complete.
	\end{proof}

	\bigskip
	
	\bibliographystyle{alpha}
	\bibliography{biblio}
	
	\bigskip
	
	{\small \textsc{UCLA Mathematics Department, Box 951555, Los Angeles, CA 90095-1555}}
	
	\textit{E-mail address:} \href{mailto:lvera@math.ucla.edu}{\texttt{lvera@math.ucla.edu}}


	
\end{document}